\documentclass[twoside]{amsart}
\usepackage{mathrsfs}
\usepackage{amsfonts}
\usepackage{amssymb}
\usepackage{amssymb}
\usepackage{amssymb}
\usepackage{amssymb}
\usepackage{amssymb}
\usepackage{amssymb}
\usepackage{amssymb}
\usepackage{amssymb}
\usepackage{amsmath}
\usepackage[all]{xy}

\usepackage{lastpage}

\RequirePackage{amsmath} \RequirePackage{amssymb}
\usepackage{amscd,latexsym,amsthm,amsfonts,amssymb,amsmath,amsxtra}
\usepackage
[colorlinks=true, 
linkcolor=blue,
urlcolor=blue,
bookmarks=true,
bookmarksopen=true,
citecolor=blue,
]{hyperref}

\newcommand{\bpm}{\begin{pmatrix}}
\newcommand{\epm}{\end{pmatrix}}
\newcommand{\bsm}{\begin{smallmatrix}}
\newcommand{\esm}{\end{smallmatrix}}
\newcommand{\bspm}{\left(\begin{smallmatrix}}
\newcommand{\espm}{\end{smallmatrix}\right)}
\newcommand{\bbm}{\begin{bmatrix}}
\newcommand{\ebm}{\end{bmatrix}}

    \newcommand{\BC}{{\mathbb {C}}} 
     \newcommand{\BF}{{\mathbb {F}}}

     \newcommand{\sH}{{\mathscr {H}}}

    \newcommand{\CA}{{\mathcal {A}}}

    \newcommand{\CS}{{\mathcal {S}}} 
     
    \newcommand{\CW}{{\mathcal {W}}}

    \newcommand{\RG}{{\mathrm {G}}}

    \newcommand{\RO}{{\mathrm {O}}}

    \newcommand{\RU}{{\mathrm {U}}}

     \newcommand{\bx}{{\bf {x}}}  \newcommand{\St}{{\mathrm {St}}}
     
      \newcommand{\bG}{{\mathrm {G}}}
\newcommand{\pr}{{\mathrm {pr}}}

     \newcommand{\Nm}{{\mathrm {Nm}}}
    \newcommand{\lenth}{{\mathrm {\lenth}}}

    \newcommand{\End}{{\mathrm{End}}} 
    \newcommand{\Gal}{{\mathrm{Gal}}} \newcommand{\GL}{{\mathrm{GL}}}
    \newcommand{\Hom}{{\mathrm{Hom}}} 
    \newcommand{\Ind}{{\mathrm{Ind}}}

      \newcommand{\Tr}{{\mathrm{Tr}}}

    \newcommand{\id}{{\mathrm{id}}}

    \newcommand{\Sp}{{\mathrm{Sp}}}  \newcommand{\Fr}{{\mathrm{Fr}}}
    \newcommand{\diag}{{\mathrm{diag}}}\newcommand{\ch}{{\mathrm{ch}}}

 \newcommand{\SL}{{\mathrm{SL}}}
 \newcommand{\SO}{{\mathrm{SO}}}

    \newcommand{\wt}{\widetilde}  
    
    \newcommand{\pair}[1]{\langle {#1} \rangle}
    \newcommand{\wpair}[1]{\left\{{#1}\right\}}

    \newcommand{\ov}{\overline}
    
    \newcommand{\incl}{\hookrightarrow}
    
     \newcommand{\ra}{\rightarrow}

    \theoremstyle{plain}

    \newtheorem{thm}{Theorem}[section] \newtheorem{cor}[thm]{Corollary}
    \newtheorem{lem}[thm]{Lemma}  \newtheorem{prop}[thm]{Proposition}
    \newtheorem {conj}[thm]{Conjecture} 
    \newtheorem{rmk}[thm]{Remark}
    
     \newtheorem*{thmA}{Theorem A}
     \newtheorem*{thmB}{Theorem B}

    \numberwithin{equation}{section}

\usepackage[top=1in, bottom=1in, left=1.25in, right=1.25in]{geometry}

\title[Uniqueness of certain Fourier-Jacobi models]{Uniqueness of certain Fourier-Jacobi models over finite fields}
\author{Baiying Liu}
\address{Department of Mathematics, Purdue University, West Lafayette, IN, USA, 47907
}
\email{liu2053@purdue.edu}
\author{Qing Zhang}
\address{Department of Mathematics and Statistics,  University of Calgary,
Calgary, Alberta, Canada, T2N 1N4 }
\email{qing.zhang1@ucalgary.ca}


\begin{document}

\maketitle
\setcounter{tocdepth}{1}

\begin{abstract}
    In this paper, we prove the uniqueness of certain Fourier-Jacobi models for the split  exceptional group $\bG_2$ over finite fields with odd characteristic. Similar results are also proved for $\Sp_4$ and $\RU_4$.
\end{abstract}


\section{Introduction}\label{sec1}
Uniqueness of Bessel models and Fourier-Jacobi models for classical groups over local fields, recently proved in \cite{AGRS,LS,Su,SZ,GGP} for various cases, has played very important roles in the study of automorphic representations and $L$-functions for classical groups. These uniqueness results are 
the starting points of the local Gan-Gross-Prasad conjectures \cite[Conjecture 17.1]{GGP}, and 
the key ingredients to construct new Rankin-Selberg integrals on these groups (see \cite{GPSR,GJRS,JZ} for some examples).
They also give local functional equations, and thus local gamma factors, for many long-known local zeta integrals for these groups (see \cite{Ka} for example). 

After proving the uniqueness of Bessel models and Fourier-Jacobi models for classical groups, a natural question to ask is whether one could define similar models for exceptional groups and prove analogous uniqueness. For the simplest exceptional group, the split group of type $\bG_2$, we can define a Fourier-Jacobi model which is quite similar to the $\Sp_4$ case. 

Let $k$ be a local field, and $\alpha,\beta$ be the two roots of $\bG_2(k)$ with $\alpha$ the short root and $\beta$ the long root, and let $P=MV$ be the parabolic subgroup of $\bG_2(k)$ with $ \beta$ in its Levi subgroup $M$, where $V$ is the unipotent part of $P$. Then one has $M\cong \GL_2(k).$ Let $J=\SL_2(k)\ltimes V\subset P$. Then one can check that there is a projection map $J\ra \SL_2(k)\ltimes \sH$, where $\sH$ is the Heisenberg group with 3 variables. For a non-trivial additive character $\psi$ of $k$, one then has a Weil representation $\omega_\psi$ of $\wt\SL_2(k)\ltimes\sH$, where $\wt\SL_2(k)$ is the metaplectic cover of $\SL_2(k)$. Thus we can view $\omega_\psi$ as a representation of $\wt J=\wt\SL_2(k)\ltimes V$. Given a genuine irreducible representation $\pi$ of $\wt\SL_2(k)$, the tensor product $\pi\otimes \omega_\psi$ gives a representation of $J$. Let $\Pi$ be an irreducible representation of $\bG_2(k)$, 
a non-trivial element in $\mathrm{Hom}_{\bG_2(k)}(\Pi, \Ind_{J}^{\bG_2(k)}(\pi\otimes\omega_\psi))$
is called a Fourier-Jacobi model of $\Pi$. 
By Frobenius reciprocity, 
$$\mathrm{Hom}_{\bG_2(k)}(\Pi, \Ind_{J}^{\bG_2(k)}(\pi\otimes\omega_\psi)) = \Hom_J(\Pi,\pi\otimes\omega_\psi).$$
 We conjecture that these Hom spaces should have dimension at most one over local fields, at least for self-dual irreducible representation $\Pi$, see Conjecture \ref{conj5.2}. It is worthwhile to mention that Ginzburg \cite{Gi} has constructed a local zeta integral which naturally lies in these Hom spaces. 

The main goal of this paper is to consider the analogue conjecture over finite fields $k$ with odd characteristic and 
verify certain uniqueness of Fourier-Jacobi models for $\bG_2(k)$. Precisely, we prove the following 

\begin{thmA}[Theorem \ref{thm6.1}]
Let $k$ be a finite field with odd characteristic. Let $\pi$ be an irreducible representation of $\SL_2(k)$ which is not fully induced from the Borel subgroup. Then $\Ind_J^{\bG_2(k)}(\pi\otimes\omega_\psi)$ is multiplicity free. 
\end{thmA}

Note that, over finite fields, the metaplectic cover $\wt{\SL}_2(k)$ splits and the Weil representation can be defined over $\SL_2(k)$. Moreover, over finite fields, every finite dimensional representation is semi-simple, and thus the above theorem is equivalent to that for any irreducible representation $\Pi$ of $\bG_2(k)$, one has $\dim\Hom_J(\Pi,\pi\otimes\omega_\psi)\le 1$ if $\pi$ is an irreducible representation of $\SL_2(k)$ which is not fully induced from the Borel subgroup. However, if $\pi$ is indeed an irreducible representation which is fully induced from the Borel subgroup, then the representation $\Ind_J^{\bG_2(k)}(\pi\otimes\omega_\psi)$ may not be multiplicity free, see Remark \ref{rmk6.2}, which is quite different from the conjectural local fields case. On the other hand, if $\pi$ is fully induced representation of $\SL_2(k)$ and if $\Pi$ is an irreducible \textit{cuspidal} representation of $\RG_2(k)$, in \cite{LZ}, we prove that $\Hom_J(\Pi,\pi\otimes \omega_\psi)\le 1$, which is used to prove the existence of $\GL_1$-twisted gamma factors.

When we were working on the above result for the split exceptional group $\bG_2$, we realized that even for classical groups over finite fields, the uniqueness of Fourier-Jacobi models has not been settled in general. Thus we decided to include some results on the uniqueness of certain Fourier-Jacobi models for $\Sp_4$ and $\RU_4$ over finite fields, which are quite similar to the $\bG_2$ case:

\begin{thmB}[Theorem \ref{thm4.1} and Theorem \ref{thm: multiplicity one for U4}]
Let $k$ be a finite field with odd characteristic and let $E/k$ be a quadratic extension. Let $G=\Sp_4(k)$ $($respectively $\RU_4(k))$ and let $J$ be a subgroup of $G$ which is isomorphic to $\SL_2(k)\ltimes \sH$ $($respectively $\RU_2(k)\ltimes \sH_E$ with $\sH_E$ the Heisenberg group of the form $E^2\oplus k)$. Let $\pi$ be an irreducible representation of $\SL_2(k)$ $($respectively $\RU_2(k)$ $)$ which is not fully induced from the Borel subgroup. Then $\Ind_J^G(\pi\otimes\omega_\psi)$ is multiplicity free.
\end{thmB}

As in the $\bG_2$ case, if $\pi$ is an irreducible representation which is fully induced from the Borel subgroup, then $\Ind_J^G(\pi\otimes\omega_\psi)$ may not be multiplicity free in general (see Remark \ref{rmk4.2}). On the other hand, since unitary groups
are inner forms of general linear groups, one might expect that similar results also hold for $\GL_4(k)$. But it turns out that even the Weil representation itself of $\GL_2(k)$ is not multiplicity free (see \cite[Proposition 4.2]{Ge}), thus we could not expect similar results for general linear groups. This also shows how tricky things can be over finite fields. 

We remark that, when $\pi$ is a fully induced representation of $\SL_2(k)$ (resp. $\RU_2(k)$) from the Borel subgroup, the fact that $\Ind_J^G(\pi\otimes\omega_\psi)$, where $G$ can be $\bG_2(k),\Sp_4(k)$ (resp. $\RU_4(k)$), may not be multiplicity free comes from the fact that $\pi\otimes\omega_\psi|_{\SL_2(k)}$ (resp. $\pi\otimes\omega_\psi|_{\RU_2(k)}$) is not multiplicity free. This is quite different from the local fields case and makes the uniqueness problem of Fourier-Jacobi models over finite fields more complicated and thus more interesting. However, in this case, we still expect irreducible cuspidal representations of $G$ occur with multiplicity one in $\Ind_J^G(\pi\otimes\omega_\psi)$ (see our subsequent work \cite{LZ}). 

Theoretically speaking, all of the above multiplicity one results could be checked using the known character tables for these groups, see \cite{Sr} for the character table of $\Sp_4(\mathbb{F}_q)$ when $q$ is odd, and \cite{CR,En,EY} for the character table of $\bG_2$. However, although the groups in our consideration are relatively small, their character tables are already too complicated to be used to give a proof of the above multiplicity one results practically. Alternatively, we use a variant of Gelfand-Kazhdan method (see Section \ref{sec3.2}) to prove the above results. In \cite{T}, using this approach, Teicher proved certain multiplicity one results for $\mathrm{GSp}_{2n}$ and $\mathrm{O}_n$ over finite fields. 

In general, let $k$ be a finite field with odd characteristic and $E/k$ be a quadratic extension,  $\bG_n(k)=\Sp_{2n}(k)$ or $\RU_{2n}(k)$, and let $\omega_\psi$ be the Weil representation of $\bG_n(k)$. A natural question is then for what kind of irreducible representations $\pi$ of $\bG_n(k)$, the tensor product $\pi\otimes\omega_\psi$ is still multiplicity free?  If $\pi$ is the trivial representation or a character, this is known to be true. If $\pi$ is the Steinberg representation, it is proved by Hiss and Zalesski (\cite{HZ}) that $\pi\otimes\omega_\psi$ is multiplicity free. It seems that little is known for general $\pi$. According to the decomposition of $\pi\otimes\omega_\psi|_{\SL_2(k)}$ given in Section \ref{sec2.5}, we guess that if $\pi$ is an irreducible cuspidal representation of $\bG_{n}(k)$, then $\pi\otimes\omega_\psi$ should be multiplicity free. This uniqueness of general Fourier-Jacobi models towards  analogues of the local Gan-Gross-Prasad conjectures over finite fields (using Deligne-Lusztig character theory)  is currently our work in progress.

These multiplicity one results have potential applications in establishing certain functional equations over finite fields and proving the existence of certain twisted gamma factors as done by Roditty in \cite{Ro} for $\GL_n$ over finite fields. In \cite{Ni}, Nien proved a local converse theorem for $\GL_n(k)$ using the local gamma factors in \cite{Ro}. We expect to construct the gamma factors (using an analogue of Rankin-Selberg method) and prove the converse theorems for other groups over finite fields in the future. 

Finally, we remark that, although few results on the multiplicities of Fourier-Jacobi models over finite fields have been obtained in the literature, 
analogue problems for Bessel models have been vastly studied, see \cite{GGP2,GP,Ha,R,Th,Za} for example. 


The paper is organized as follows. In Section 2, we give a review of representations of $\SL_2(k)$ over finite fields. In Section 3 and 4, we prove Theorem B
for $\Sp_4(k)$. The $\RU_4(k)$ case is included in Section 5. We then introduce the Fourier-Jacobi models for $\bG_2(k)$ in Section 6, and prove Theorem A in Section 7. 

\section*{Acknowledgement}
The first-named author is partially supported by NSF grant DMS-1702218 and start-up funds from the Department of Mathematics at Purdue University. The second-named author was partially supported by Guangdong Natural Science Foundation 2018A030310081, NSFC grant 11801577 and a PIMS posdtoc fellowship. The authors would like to thank Professor Meinolf Geck for useful discussion and pointing out the reference \cite{HZ}. Part of the work was finished when the second-named author was a research associate in Sun Yat-Sen University, Guangzhou, China, visited Bin Xu in Sichuan University and attended the workshop ``Introductory Workshop: Group Representation Theory and Applications" at MSRI, where he had a chance to talk to Professor Meinolf Geck, he would like to thank the support from these institutes.

\section{Review of representations of SL(2) over finite fields}
In this section, we collect some well-known facts on representations of $\SL_2$ over finite fields with odd characteristic.

Throughout the paper, unless otherwise specified, we let $q=p^r$ with an odd prime $p$, and let $k=\BF_q$, the unique (up to isomorphism) finite field with $q$ elements. Let $\psi$ be a fixed non-trivial additive character of $k$.  Write $\epsilon_0=\left( \frac{-1}{q}\right)$ and $\epsilon(x)=\left(\frac{x}{q} \right)$ for $x\in k^\times$, where $\left(\frac{\cdot}{q} \right)$ is the Legendre symbol. Note that $\epsilon$ is the unique non-trivial quadratic character of $k^\times$ and the kernel of $\epsilon$ is $k^{\times,2}:=\wpair{x^2:x\in k^\times}$. We fix a generator $\kappa$ of the cyclic group $k^\times$. Then a set of representatives of $k^\times/k^{\times,2}$ can be taken as $\wpair{1,\kappa}$. 

 Let $E$ be the unique (up to isomorphism) quadratic extension of $k$. Let $\Fr: E\ra E$ be the Frobenius map defined by $\Fr(x)=x^q$. Then $\Fr$ is the unique non-trivial element in the Galois group $\Gal(E/k)$. Let $\Nm:E\ra k$ be the norm map and let $\Tr:E\ra k$ be the trace map. Note that $\Nm(x)=x\cdot \Fr(x)=x^{q+1}$ and $\Tr(x)=x+x^q.$  We can realize $E$ as $k[\sqrt{\kappa}]$. Under this realization, we have $ \Fr(x+y\sqrt{\kappa})=x-y\sqrt{\kappa},$ $ \Nm(x+y\sqrt{\kappa})=x^2-y^2\kappa,$ and
     $\Tr(x+y\sqrt{\kappa})=2x,$ for $x,y\in k.$ Let $E^1=\wpair{x\in E^\times: \Nm(x)=1}$. Then the norm map $\Nm:E^\times \ra k^\times$ induces an exact sequence 
$1\ra E^1\ra E^\times\ra k^\times\ra 1.$

\subsection{Conjugacy classes and induced representations of SL(2)}\label{sec2.1}

Table \ref{table: conjugacy class of SL(2)} gives the conjugacy classes of $\SL_2(k)$ (see \cite[$\S$5]{FH}).

\begin{table}
\begin{align*}
\begin{array}{|c|c|c|}
\hline
\textrm{Representative} & \textrm{Number of elements in class } & \textrm{ Number of classes }   \\
\hline
\begin{pmatrix}1 &\\ &1 \end{pmatrix}  &  1&1   \\
\hline
\begin{pmatrix}-1 &\\ &-1 \end{pmatrix}  & 1 &1  \\
\hline
 \begin{pmatrix}1 &1\\ &1\end{pmatrix}&(q^2-1)/2&1\\
 \hline
 \begin{pmatrix}1 &\kappa\\ &1\end{pmatrix}&(q^2-1)/2&1\\
 \hline
\begin{pmatrix}-1 &1\\ &-1\end{pmatrix}&(q^2-1)/2&1\\
\hline
 \begin{pmatrix}-1 &\kappa\\ &-1\end{pmatrix}&(q^2-1)/2&1\\
 \hline
 \begin{pmatrix}x & \\ &x^{-1} \end{pmatrix}, x\ne \pm1 &q(q+1) & (q-3)/2\\
 \hline
\begin{pmatrix}x&y\\ \kappa y&x \end{pmatrix},x\ne \pm1, y \ne 0 &q(q-1) &(q-1)/2\\
\hline
\end{array}
\end{align*}
\caption{Conjugacy class of $\SL_2(k)$}
\label{table: conjugacy class of SL(2)}
\end{table}

In the second to the last row in Table \ref{table: conjugacy class of SL(2)}, $\begin{pmatrix}x & \\ &x^{-1} \end{pmatrix}$ and $\begin{pmatrix}x^{-1} & \\ &x \end{pmatrix}$ are in the same class. 
In the last row, $\begin{pmatrix}x&y\\ \kappa y&x \end{pmatrix}$ and $\begin{pmatrix}x&-y\\ -\kappa y&x \end{pmatrix}$ are in the same class, and  $\begin{pmatrix}x&y\\ \kappa y&x \end{pmatrix}$ can be mapped to $x+y\sqrt{\kappa}\ne \pm 1$ in $E^1$. Thus these representatives can be indexed by $(E^1-\wpair{\pm 1})/\pair{\mathrm{Fr}}.$ For  $\begin{pmatrix}x&y\\ \kappa y&x \end{pmatrix}\in \SL_2(k)$, we write $\xi_{x,y}=x+y\sqrt{\kappa}$.

The simplest class of irreducible representations of $\SL_2(k)$ is that induced from characters of the Borel subgroup. Let $B_{\SL_2}=A_{\SL_2}N_{\SL_2}$ be the Borel subgroup of $\SL_2(k)$ consisting of upper triangular matrices, with diagonal torus $A_{\SL_2}\cong k^\times$ and upper triangular unipotent subgroup $N_{\SL_2}$. Let $\chi$ be a character of $k^\times\cong A_{\SL_2}$. We can view $\chi$ as a character of $B_{\SL_2}$ such that its action on $N_{\SL_2}$ is trivial. Let $I(\chi)$ be the induced representation $\Ind_{B_{\SL_2}}^{\SL_2(k)}(\chi)$, then 
$\dim I(\chi)=q+1$. If $\chi^2\ne 1$, then it is well-known that $I(\chi)$ is irreducible (see \cite[$\S$2.5]{Pr} for example),  $I(\chi)\cong I(\chi^{-1})$, and there are totally $\frac{q-3}{2}$ of them. If $\chi=1$, the trivial character, then $I(1)=1\oplus \St$, where $1$ denotes the trivial representation of $\SL_2(k)$ and $\St$ denotes the Steinberg representation. Hence, $\dim \St=q$.
If $\chi = \epsilon$, then $I(\epsilon)$ is a direct sum of two non-equivalent irreducible representations, each of which is a Weil representation (cf. the next subsection), and has dimension $\frac{q+1}{2}$, 

\subsection{Weil representations}\label{sec2.2}
Since $N_{\SL_2}\cong k$, we could view $\psi$ as a character of $N_{\SL_2}$. Up to conjugation by $A_{\SL_2}$, there are two non-trivial additive characters of $N_{\SL_2}$, which are $\psi$ and $\psi_\kappa$, where $\psi_\kappa$ is the character given by $\psi_\kappa(x)=\psi(\kappa x)$.

Let $W=k^2$, endowed with the symplectic structure $\pair{~,~}$ defined by
$$\pair{(x_1,y_1),(x_2,y_2)}=2x_1y_2-2x_2y_1. $$
Let $\sH$ be the Heisenberg group associated with $W$. Explicitly, $\sH=W\oplus k$ with addition 
$$[x_1,y_1,z_1]+[x_2,y_2,z_2]=[x_1+x_2,y_1+y_2,z_1+z_2+x_1y_2-x_2y_1].$$
Let $\SL_2(k)$ act on $\sH$ such that it acts on $W$ by right multiplication and acts on the third component $k$ in $\sH$ trivially. Then we can form the semi-direct product $\SL_2(k)\ltimes \sH.$ It is well-known that there is a Weil representation $\omega_\psi$ of $\SL_2(k)\ltimes \sH$ on $\CS(k)$, where $\CS(k)$ is the space of $\BC$-valued functions on $k$. This representation is determined by the following formulas 
\begin{align}
\begin{split}\label{eq2.1}
\omega_\psi([x,0,z])\phi(\xi)&=\psi(z)\phi(\xi+x),\\
\omega_\psi([0,y,0])\phi(\xi)&=\psi(2\xi y)\phi(\xi),\\
\omega_\psi(\diag(a,a^{-1}))\phi(\xi)&=\epsilon(a)\phi(a\xi),\\
\omega_\psi\left(\begin{pmatrix}1&b\\ &1 \end{pmatrix} \right)\phi(\xi)&=\psi(b\xi^2)\phi(\xi),\\
\omega_\psi\left(\begin{pmatrix}&b\\ -b^{-1}& \end{pmatrix} \right)\phi(\xi)&=\frac{1}{\gamma(b,\psi)}\sum_{x\in k}\psi(2bx\xi)\phi(x),
\end{split}
\end{align}
where $\gamma(b,\psi)=\sum_{x\in k}\psi(-bx^2)$. The construction of Weil representations for more general groups over finite fields was given in \cite{Ge} after the seminal work of Weil \cite{W}. The above formulas also could be found in \cite[p. 220]{GH}. 

The representation $\omega_\psi|_{\SL_2(k)}$ is reducible.  In fact, let
$$\CS^\pm(k)=\wpair{\phi\in \CS(k): \phi(-\xi)=\pm\phi(\xi),\forall \xi\in k}.$$ 
Then $\CS^+(k)$ and $\CS^-(k)$ are invariant under the action of $\SL_2(k)$. Denote the corresponding representations by $\omega_\psi^+$ and $\omega_\psi^-$ respectively. Then $\omega_\psi^+$ and $\omega_\psi^-$ are irreducible, $\dim \omega_\psi^+=\frac{q+1}{2}$ and $\dim \omega_{\psi}^-=\frac{q-1}{2}$. 
Similarly, from the character $\psi_{\kappa}$, one can construct the Weil representation $\omega_{\psi_\kappa}=\omega_{\psi_\kappa}^+ \oplus \omega_{\psi_\kappa}^-$.
Moreover, we have 
$I(\epsilon)=\omega_\psi^+\oplus \omega_{\psi_\kappa}^+.$
One can also check that $\omega_\psi^-$ and $\omega_{\psi_{\kappa}}^-$ are cuspidal, in the sense that they are not subrepresentations of $I(\chi)$ for any character $\chi$ of $k^\times$. 

\subsection{Cuspidal representations}\label{cuspidal representations}
We already have two cuspidal representations $\omega_{\psi}^-,\omega_{\psi_\kappa}^-$. The remaining cuspidal representations are also constructed from Weil representations.

Let $\mu$ be a non-trivial character of $E^1$. Let 
$$\CW(\mu)=\wpair{f:E \ra \BC: f(yx)=\mu^{-1}(y)f(x),\forall x\in E, y\in E^1}.$$
For $f\in \CW(\mu)$, we have $f(0)=\mu^{-1}(y)f(0)$, for all $y\in E^1$. Since $\mu$ is non-trivial, we get $f(0)=0$. For any $a\in k^\times$, let $x_a\in E^\times$ be such that $\Nm(x_a)=a$. Note that a function $f\in \CW(\mu)$ is uniquely determined by its values on the set $\wpair{x_a:a\in k^\times}$. Thus $\dim \CW(\mu)=q-1$.

There is a representation $\omega_{\psi,\mu}$ of $\SL_2(k)$ on $\CW(\mu)$ such that 
\begin{align}
\begin{split}\label{eq2.2}
\omega_{\psi,\mu}\left( \begin{pmatrix}a&\\ &a^{-1} \end{pmatrix}\right)f(\xi)&=f(a\xi),\\
\omega_{\psi,\mu}\left( \begin{pmatrix}1&b\\ &1 \end{pmatrix}\right)f(\xi)&=\psi(\Nm(\xi)b)f(\xi),\\
\omega_{\psi,\mu}\left( \begin{pmatrix}&1\\ -1& \end{pmatrix}\right)f(\xi)&=-q^{-1}\sum_{y\in E}\psi(\Tr(\Fr( y) \xi))f(y).
\end{split}
\end{align}
See \cite[$\S$4.1]{Bu}, or \cite[Chapter 3]{Pr}.

The representation $\omega_{\psi,\mu}$ is cuspidal. If $\mu^2\ne 1$, then $\omega_{\psi,\mu}$ is irreducible. Furthermore, $\omega_{\psi,\mu}$ is isomorphic to $\omega_{\psi,\mu'}$ if and only if $\mu'=\mu^{\pm 1}$.
One can also check that $\omega_{\psi,\mu}\cong \omega_{\psi_\kappa,\mu}$, and thus the representation $\omega_{\psi,\mu}$ is independent of the choice of $\psi$.
We then get $\frac{q-1}{2}$ equivalence classes in the family $\wpair{\omega_{\psi,\mu},\mu\in \widehat E^1,\mu^2\ne 1}$. 
Let $\mu_0$ be the unique non-trivial quadratic character of $E^1$, then one can check that $\omega_{\psi,\mu_0}=\omega_\psi^-\oplus \omega_{\psi_\kappa}^-$.

So far, we get a list of irreducible representations
$$1, \St, I(\chi)~~ ( \chi^2\ne 1), \omega_\psi^{\pm},\omega_{\psi_\kappa}^{\pm}, \omega_{\psi,\mu}~~(\mu^2\ne 1),$$
and the only non-trivial relations among them are $I(\chi)\cong I(\chi^{-1})$ and $ \omega_{\psi,\mu}\cong \omega_{\psi,\mu^{-1}}$. One can easily see that this is a complete list of irreducible representations of $\SL_2(k)$ by checking that the cardinality of this list is exactly the same as the number of conjugacy classes of $\SL_2(k)$. 

\subsection{Character table}\label{sec2.4}

For a finite dimensional representation $\rho$ of a finite group $G$, we denote by $\ch_\rho$ the character function of $G$. Recall that $\ch_\rho(g)=\Tr_\rho(g)=\sum_{i}r_{ii}$, if $\rho(g)$ is identified with a matrix $(r_{ij})$. The function $\ch_\rho$ is a class function, i.e., it is constant on a conjugacy class. Table \ref{table: character table of SL(2)} gives the character table of $\SL_2(k)$, which is taken from \cite[$\S$5]{FH}:

\begin{table}
\begin{align*}
\begin{array}{|c|c|c|c|c|}
\hline
\ch& \begin{pmatrix}x &\\ &x \end{pmatrix},x=\pm1 & \begin{pmatrix}x& y \\&x \end{pmatrix},
\begin{matrix}
 x=\pm 1\\
 y=1,\kappa 
\end{matrix}
& \begin{pmatrix}x&\\ &x^{-1} \end{pmatrix}, x\ne \pm1& \begin{pmatrix} x&  y \\ \kappa y &x\end{pmatrix},
\begin{matrix}
 x\ne \pm1\\
 y\ne 0
\end{matrix}
\\
\hline
1 &1&1&1&1\\
\hline
{\St}&q&0&1&-1\\
\hline
{I(\chi)}& (q+1)\chi(x)& \chi(x)&\chi(x)+\chi(x^{-1})& 0\\
\hline
{\omega_{\psi,\mu}},\mu^2\ne 1& (q-1)\mu(x)&-\mu(x)&0&-(\mu+\mu^q)(\xi_{x,y}), \\
\hline
\omega_{\psi}^+,\omega_{\psi_\kappa}^+&&&\epsilon(x)&0\\
\hline
\omega_{\psi}^-,\omega_{\psi_\kappa}^-&&&0&-\mu_0(\xi_{x,y})\\
\hline
\end{array}
\end{align*}
\caption{Character table of $\SL_2(k)$}
\label{table: character table of SL(2)}
\end{table}
The missing part of Table \ref{table: character table of SL(2)} is given in Table \ref{table: missing part}.
\begin{table}
\begin{align*}
\begin{array}{|c|c|c|c|c|c|c|}
\hline
&\begin{pmatrix}1&\\&1 \end{pmatrix} & \begin{pmatrix}-1&\\ &-1 \end{pmatrix} &\begin{pmatrix}1&1\\ &1 \end{pmatrix}&\begin{pmatrix}1&\kappa\\ &1 \end{pmatrix}&\begin{pmatrix}-1&1\\ &-1 \end{pmatrix} & \begin{pmatrix}-1&\kappa\\ &-1 \end{pmatrix}\\
\hline
\omega_\psi^+ &\frac{q+1}{2}&\frac{q+1}{2}\epsilon_0 &s&t&s'&t'\\
\hline
\omega_{\psi_\kappa}^+&\frac{q+1}{2}&\frac{q+1}{2}\epsilon_0&t&s&t'&s'\\
\hline
\omega_\psi^-&\frac{q-1}{2}&-\frac{q-1}{2}\epsilon_0&u&v&u'&v'\\
\hline
\omega_{\psi_\kappa}^-&\frac{q-1}{2}&-\frac{q-1}{2}\epsilon_0 &v&u&v'&u'\\
\hline
\end{array}
\end{align*}
\caption{Missing part of Table \ref{table: character table of SL(2)}}
\label{table: missing part}
\end{table}
The parameters $s,t,u,v,s',t',u',v'$ in Table \ref{table: missing part} are given below:
\begin{align*}
s=\frac{1}{2}+\frac{1}{2}\sqrt{ \epsilon_0 q}, &\quad t=\frac{1}{2}-\frac{1}{2}\sqrt{\epsilon_0 q},\\
u=-\frac{1}{2}+\frac{1}{2}\sqrt{\epsilon_0 q}, &\quad v=-\frac{1}{2}-\frac{1}{2}\sqrt{\epsilon_0 q},\\
s'=\epsilon_0 s,
t'=\epsilon_0 t, &\quad  u'=t'=\epsilon_0 t, v'=s'=\epsilon_0 s.
\end{align*}

\subsection{The tensor product \texorpdfstring{$\pi\otimes\omega_{\psi}$}{Lg}} \label{sec2.5}

Given an irreducible representation $\pi$ of $\SL_2(k)$, we want to determine the decomposition of $\pi\otimes \omega_{\psi}$. We first notice that $\ch_{\omega_\psi}=\ch_{\omega_{\psi}^+}+\ch_{\omega_\psi^-}$. Thus we have 
\begin{align*}
\begin{array}{|c|c|c|c|c|c|c|c|c|}
\hline
\ch &\begin{pmatrix}1 &\\ &1\end{pmatrix} &\begin{pmatrix}-1 &\\ &-1\end{pmatrix}&\begin{pmatrix}1 &1\\ &1\end{pmatrix}&\begin{pmatrix}1 &\kappa\\ &1\end{pmatrix}&\begin{pmatrix}-1 &1\\ &-1\end{pmatrix}&\begin{pmatrix}-1 &\kappa\\ &-1\end{pmatrix}&\begin{pmatrix}x &\\ &x^{-1}\end{pmatrix}&\begin{pmatrix}x &y\\y\kappa &x\end{pmatrix}\\
\hline
\omega_\psi&q&\epsilon_0 &\sqrt{\epsilon_0 q}&-\sqrt{\epsilon_0 q}&\epsilon_0&\epsilon_0&\epsilon(x)&-\mu_0(\xi_{x,y} )\\
\hline
\end{array}
\end{align*}

Let $\pi$ be an irreducible representation of $\SL_2(k)$, we are going to compute $\ch_{\pi\otimes \omega_\psi}$. It is well-known that $\ch_{\pi\otimes \omega_\psi}(g)=\ch_\pi(g)\ch_{\omega_\psi}(g)$.  Hence, we have the following character table


\begin{align*}
\begin{array}{|c|c|c|c|c|c|c|c|}
\hline
\ch &\St\otimes\omega_\psi &I(\chi)\otimes\omega_\psi  &\omega_{\psi,\mu}\otimes\omega_\psi
&\omega_\psi^+\otimes\omega_\psi& \omega_{\psi_\kappa}^+\otimes \omega_\psi & \omega_\psi^-\otimes \omega_\psi&\omega_{\psi_\kappa}^-\otimes \omega_\psi\\
\hline
\begin{pmatrix}1&\\ &1 \end{pmatrix}&q^2 &q(q+1) &q(q-1) &\frac{q(q+1)}{2}&\frac{q(q+1)}{2}&\frac{q(q-1)}{2}&\frac{q(q-1)}{2}\\
\hline
\begin{pmatrix}-1&\\ &-1 \end{pmatrix}&\epsilon_0q&(q+1)\chi(-1)\epsilon_0& (q-1)\mu(-1)\epsilon_0& \frac{q+1}{2}&\frac{q+1}{2}&\frac{-q+1}{2}&\frac{-q+1}{2}\\
\hline
\begin{pmatrix}1&1\\ &1 \end{pmatrix}&0&\sqrt{\epsilon_0q}&-\sqrt{\epsilon_0q}
&s\sqrt{\epsilon_0 q}&t\sqrt{\epsilon_0 q}&u\sqrt{\epsilon_0 q}&v\sqrt{\epsilon_0 q}\\
\hline
\begin{pmatrix}1&\kappa\\ &1 \end{pmatrix}& 0&-\sqrt{\epsilon_0 q}&\sqrt{\epsilon_0q}
&-t\sqrt{\epsilon_0 q}&-s\sqrt{\epsilon_0 q}&-v\sqrt{\epsilon_0 q}&-u\sqrt{\epsilon_0 q}\\
\hline
\begin{pmatrix}-1&1\\ &-1 \end{pmatrix}&0&\epsilon_0\chi(-1)&-\epsilon_0\mu(-1)
&s&t&t&s\\
\hline
\begin{pmatrix}-1&\kappa\\ &-1 \end{pmatrix}&0 &\epsilon_0\chi(-1)&-\epsilon_0\mu(-1)
&t&s&s&t\\
\hline
\begin{pmatrix}x&\\ &x^{-1} \end{pmatrix}&\epsilon(x) &\epsilon\chi(x)+\epsilon\chi(x^{-1})&0
&1&1&0&0\\
\hline
\begin{pmatrix}x&y\\y\kappa  &x \end{pmatrix} &\mu_0(\xi_{x,y})&0 &(\mu+\mu^q)\mu_0(\xi_{x,y})
&0&0&1&1\\
\hline
\end{array}
\end{align*}
where in the last two rows $x\ne \pm 1$, and in the last row $y\ne 0$ as usual.

Let $A\subset \widehat {k^\times}$ be a set of representatives of $(\widehat {k^\times} -\wpair{1,\epsilon})/\pair{\chi=\chi^{-1}}$, where $1$ is the trivial character. In other words, for each $\chi\in \widehat {k^\times}-\wpair{ 1,\epsilon}$, there is one and only one of $\chi,\chi^{-1}$ is in $A$. Thus $|A|=\frac{q-3}{2}$. Let $B$ be a set of representatives of $(\widehat {E^1}-\wpair{1,\mu_0})/\pair{\mu=\mu^{-1}}$. Then $|B|=\frac{q-1}{2}$.

\begin{prop}\label{prop2.1}
 Let $\chi_1$ $($resp. $\mu_1)$ be a character of $k^\times$ $($resp. $E^1)$ such that $\chi_1^2\ne 1$ $($resp. $\mu_1^2\ne 1)$.
We have the following decomposition of representations of $\SL_2(k)$:
\begin{align*}\St\otimes \omega_\psi&=\St\bigoplus\left(\bigoplus_{\chi\in A}I(\chi)\right)\bigoplus\left(\bigoplus_{\mu\in B}\omega_{\psi,\mu} \right)\bigoplus \omega_{\psi}^+ \bigoplus \omega_{\psi_\kappa}^+,\\
I(\chi_1)\otimes \omega_{\psi}&=\St\bigoplus 2I(\epsilon \chi_1)\bigoplus \left(\bigoplus_{\chi\in A,\chi\ne \epsilon \chi_1^{\pm}}I(\chi)\right)\bigoplus \left(\bigoplus_{\mu\in B}\omega_{\psi,\mu} \right)\\
&\qquad  \bigoplus \frac{1+\epsilon_0}{2}(\omega_\psi^+\oplus \omega_\psi^-) \bigoplus \frac{1-\epsilon_0}{2}(\omega_{\psi_\kappa}^+\oplus \omega_{\psi_\kappa}^-),\\
\omega_{\psi,\mu_1}\otimes \omega_\psi&=\St\bigoplus  \left(\bigoplus_{\chi\in A}I(\chi)\right)\bigoplus  \left(\bigoplus_{\mu\in B,\mu\ne \mu_0\mu_1^{\pm1}}\omega_{\psi,\mu} \right)\\
&\qquad \bigoplus \frac{1-\epsilon_0}{2}(\omega_\psi^+\oplus \omega_\psi^-) \bigoplus \frac{1+\epsilon_0}{2}(\omega_{\psi_\kappa}^+\oplus \omega_{\psi_\kappa}^-),\\
\omega_\psi^+\otimes \omega_\psi&= \St\bigoplus   \frac{1+\epsilon_0}{2}\left( 1\oplus \bigoplus_{\chi\in A}I(\chi)\bigoplus \omega_\psi^+ \right)\bigoplus \frac{1-\epsilon_0}{2}\left(\bigoplus_{\mu\in B}\omega_{\psi,\mu}\bigoplus\omega_{\psi_\kappa}^- \right),\\
\omega_{\psi}^-\otimes\omega_\psi&=\frac{1+\epsilon_0}{2}\left(1\oplus \bigoplus_{\chi\in A}I(\chi)\bigoplus \omega_{\psi_\kappa}^+\right)\bigoplus \frac{1-\epsilon_0}{2}\left(\bigoplus_{\mu\in B}\omega_{\psi,\mu}\bigoplus \omega_\psi^- \right),\\
\omega_{\psi_\kappa}^+\otimes \omega_\psi&=\St\bigoplus \frac{1-\epsilon_0}{2}\left(1\oplus \bigoplus_{\chi\in A}I(\chi)\bigoplus \omega_{\psi_\kappa}^+ \right)\bigoplus \frac{1+\epsilon_0}{2}\left(\bigoplus_{\mu\in B}\omega_{\psi,\mu}\bigoplus\omega_{\psi}^-\right),\\
\omega_{\psi_\kappa}^-\otimes\omega_\psi&=\frac{1-\epsilon_0}{2}\left(1\oplus \bigoplus_{\chi\in A}I(\chi)\bigoplus \omega_{\psi}^+\right)\bigoplus \frac{1+\epsilon_0}{2}\left(\bigoplus_{\mu\in B}\omega_{\psi,\mu}\bigoplus \omega_{\psi_\kappa}^- \right).
\end{align*}
\end{prop}
\begin{proof}
Denote by $\wpair{V_i}$ a complete set of irreducible representations of $\SL_2(k)$. Then any representation $V$ of $\SL_2(k)$ can be written as 
$$V=\oplus m_i V_i.$$
We have $m_i=(\ch_V,\ch_{V_i})$, where $$(\ch_V,\ch_{V_i})=\frac{1}{|\SL_2(k)|}\sum_{g\in \SL_2(k)}\ch_V(g)\overline{\ch_{V_i}(g)}$$
is the standard inner product of $\ch_V$ and $\ch_{V_i}$. Thus to prove the proposition, one needs to compute various $(\ch_{\pi\otimes \omega_\psi},\ch_{\pi'})$ for any pair of irreducible representations $\pi,\pi'$ of $\SL_2(k)$ using the above tables. In the following, we only give the computation of $(\ch_{I(\chi_1)\otimes \omega_\psi},\ch_{I(\chi)})$ for a character $\chi$ with $\chi^2\ne 1$, and omit similar calculations of other cases. 

By the definition of the standard inner product, we have 
\begin{align*}
    &\ |\SL_2(k)|(\ch_{I(\chi_1)\otimes \omega_\psi},\ch_{I(\chi)})\\
    =&\ q(q+1)^2+(q+1)^2\chi_1\chi(-1)\epsilon_0+\frac{q^2-1}{2}\sqrt{\epsilon_0 q}+\frac{q^2-1}{2}(-\sqrt{\epsilon_0 q}) \\
    & +\frac{q^2-1}{2}\chi_1\chi(-1)\epsilon_0+\frac{q^2-1}{2}\chi_1\chi(-1)\epsilon_0 \\
    & +q(q+1)\sum_{x\in (k^\times-\wpair{\pm 1})/(\pm 1)}(\chi(x)+\chi(x^{-1}))(\epsilon\chi_1(x)+\epsilon\chi_1^{-1}(x))\\
    = & \ q(q+1)^2+2q(q+1)\chi_1\chi\epsilon(-1)+q(q+1)\sum_{x\in k^\times-\wpair{\pm 1}}(\chi\chi_1\epsilon(x)+\epsilon\chi^{-1}\chi_1(x)).
\end{align*}
Since $\chi(-1)=\chi^{-1}(-1)$ and $q(q+1)^2=q(q+1)(q-1)+2q(q+1)$, we get 
\begin{align*}
    &\ |\SL_2(k)|(\ch_{I(\chi_1)\otimes \omega_\psi},\ch_{I(\chi)})\\
    =&\ q(q+1)(q-1)+q(q+1)\sum_{x\in k^\times}(\chi\chi_1\epsilon(x)+\chi^{-1}\chi_1\epsilon(x)).
\end{align*}
Note that when $\chi\ne \epsilon \chi_1^{\pm}$, we have $ \sum_{x\in k^\times}(\chi\chi_1\epsilon(x)+\chi^{-1}\chi_1\epsilon(x))=0$, and if $\chi=\epsilon \chi_1,$ or $\chi=\epsilon\chi_1^{-1}$, we have $\sum_{x\in k^\times}(\chi\chi_1\epsilon(x)+\chi^{-1}\chi_1\epsilon(x))=q-1$. Since $|\SL_2(k)|=q(q+1)(q-1)$, we conclude that $ (\ch_{I(\chi_1)\otimes \omega_\psi},\ch_{I(\chi)})=1$ if $\chi\ne \epsilon \chi_1^{\pm 1}$, and $(\ch_{I(\chi_1)\otimes \omega_\psi},\ch_{I(\chi)})=2$ if $\chi=\epsilon\chi_1^{\pm 1}$.
\end{proof}

\begin{cor}\label{cor2.2} Let $\pi$ be an irreducible representation $\SL_2(k)$.
If $\pi\ne I(\chi_1)$ with $\chi_1^2\ne 1$, then the tensor product $\pi\otimes \omega_\psi$ is multiplicity free.
\end{cor}
\begin{proof}
The assertion follows from Proposition \ref{prop2.1} and the fact that $\frac{1+\epsilon_0}{2},\frac{1-\epsilon_0}{2}\in \wpair{0,1}$.
\end{proof}

 \begin{rmk} \label{rmk2.3}
 \rm{Let $k$ be a $p$-adic field, $n$ be a positive integer. Let $\wt{\Sp}_{2n}(k)$ be the metaplectic cover of the symplectic group $\Sp_{2n}(k)$, and $\omega_\psi$ be the Weil representation of $\wt{\Sp}_{2n}(k)$. Then by the main theorem of \cite{Su}, for any irreducible smooth genuine representation $\pi$ of $\wt{\Sp}_{2n}(k)$, the tensor product representation $\pi\otimes \omega_{\psi}$ of $\Sp_{2n}(k)$ is multiplicity free. While when $k$ is a finite field, Proposition \ref{prop2.1} shows that $\Hom_{\SL_2(k)}(I(\epsilon \chi_1),I(\chi_1)\otimes \omega_\psi)=2$. This shows that the multiplicity one result of $\pi\otimes \omega_\psi$ fails in general over finite fields. This kind of phenomenon is also known for Bessel models over finite fields (see \cite{GGP2,R} for examples).  \qed}
 \end{rmk}

\begin{rmk} \label{rmk2.4}
\rm{Over finite fields, few results are known on the multiplicities of decomposition of $\pi\otimes \omega_\psi|_{\Sp_{2n}(k)}$ for general $n$ and general irreducible representations $\pi$ of $\Sp_{2n}(k)$. To the author's knowledge, only when $\pi$ is the Steinberg representation of $\Sp_{2n}(k)$, it is shown in \cite[Corollary 1.3]{HZ} that $\pi\otimes\omega_{\psi}$ is multiplicity free. Based on the $\SL_2(k)$ case given in Proposition \ref{prop2.1} and the corresponding results for Bessel models \cite{GGP2,R}, one might guess that $\pi\otimes\omega_\psi|_{\Sp_{2n}(k)}$ should be multiplicity free when $\pi$ is an irreducible cuspidal  representation on $\Sp_{2n}(k)$.\qed}
\end{rmk}

\subsection{Dual representation}\label{sec1.6}
For $g\in \SL_2(k)$, we define ${}^\iota\! g=d_1gd_1$, where $d_1=\diag(-1,1)$. Then ${}^\iota$ is an involution of $\SL_2(k)$. Given an irreducible representation $\pi$ of $\SL_2(k)$, let ${}^\iota\! \pi$ be the representation such that its space is the same with $\pi$ and its action is given by ${}^\iota \!\pi(g)=\pi({}^\iota\! g)$. For later use, we record the following result 
\begin{lem}\label{lem2.5}
Let $\pi$ be an irreducible representation of $\SL_2(k)$, and $\tilde \pi$ be its dual representation. Then we have $\tilde \pi\cong {}^\iota\! \pi.$
\end{lem}
\begin{proof}
We have $\ch_{\tilde \pi}(g)=\ch_{\pi}(g^{-1})$ and $\ch_{{}^\iota\!\pi}(g)=\ch_{\pi}({}^\iota\! g)$. From the character table given in Section \ref{sec2.4}, we can check case by case that $\ch_{\tilde \pi}=\ch_{{}^\iota\! \pi}$. Thus $\tilde \pi \cong {}^\iota\! \pi.$ 
\end{proof}
The involution ${}^\iota$ is called an MVW involution of $\SL_2(k)$ and can be defined in a more general setting, see \cite[p.91]{MVW}.

\section{Construction of transpose on certain End spaces}\label{sec3}
In this section, we introduce the multiplicity one problem of certain Fourier-Jacobi models on $\Sp_4$ and general strategies to attack such a  problem. We then construct transpose operators on certain End spaces as preparations for proving the multiplicity one theorems for $\Sp_4$, $\bG_2$, and $\mathrm{U}_4$.

\subsection{The problem}  For a positive integer $n$, let $J_n$ be the matrix defined by 
$$J_n=\bpm &1\\ J_{n-1}&\epm, J_1=(1).$$
 Let  $$\Sp_{2n}(k)=\wpair{g\in \GL_{2n}(k)|  g \bpm &J_n\\-J_n & \epm {}^t\!g=\bpm &J_n\\ -J_n &\epm}.$$
Note that $\Sp_2(k)=\SL_2(k)$. We will mainly focus on $\Sp_4(k)$ in this section. A typical element in the torus of $\Sp_4(k)$ has the form $t=\diag(a,b,b^{-1},a^{-1})$, $a,b\in k^\times$. Let $\alpha,\beta$ be the two simple roots defined by 
$$\alpha(t)=a/b, \quad \beta(t)=b^2, \textrm{ for } a,b\in k^\times.$$
Denote 
$$s_\alpha=\begin{pmatrix}&1&&\\ -1&&&\\ &&&-1\\ &&1& \end{pmatrix}, \textrm{ and } s_\beta=\begin{pmatrix}1&&&\\ &&1&\\ &-1&&\\ &&&1 \end{pmatrix}.$$
Then $s_\alpha, s_\beta$ are representatives of the reflections defined by $\alpha$ and $\beta$, respectively. We write 
$$\bx_\beta(b)=\begin{pmatrix}1&&&\\ &1&b&\\ &&1&\\ &&&1 \end{pmatrix},b\in k.$$

Let $P=MU$ be the parabolic subgroup of $\Sp_4(k)$ with Levi subgroup 
$$M=\wpair{\begin{pmatrix}a&&\\ &g&\\ &&a^{-1} \end{pmatrix},a\in k^\times, g\in \SL_2(k)},$$
and unipotent subgroup 
$$U=\wpair{\begin{pmatrix}1&*&*&*\\ &1&0&*\\ &&1&*\\ &&&1 \end{pmatrix}\in \Sp_4(k)}.$$
Let $M_0=\wpair{\diag(1,g,1),g\in \SL_2(k)}$ and let $J=M_0U\subset P$. We view $\SL_2(k)$ as a subgroup of $\Sp_4(k)$ via $\SL_2(k)\incl M_0$.

There is an isomorphism $\SL_2(k)\ltimes\sH\ra J$ defined by 
$$(g,[v,z])\ra \begin{pmatrix}1&v&z\\ &g&v^*\\ &&1 \end{pmatrix},$$
where $v=(x,y)\in W$ and $v^*=\begin{pmatrix}y\\ -x \end{pmatrix}$. We identify $[x,y,z]$ as an element in $J$ in this way. Thus the Weil representation $\omega_\psi$ of $\SL_2(k)\ltimes \sH$ gives a representation of $J$. Given an irreducible representation $\pi$ of $\SL_2(k)$, we consider the representation $\pi\otimes \omega_\psi$ of $J$. The action of $\pi\otimes \omega_\psi$ is given by $\pi\otimes \omega_\psi(j)(v_1\otimes v_2)=\pi(p(j))v_1\otimes \omega_\psi(j)v_2$, for $j\in J,v_1\in \pi,v_2\in \omega_\psi$, where $p:J\ra \SL_2(k)$ is the natural projection. It is known that $\pi\otimes \omega_\psi$ is irreducible as a representation of $J$, see \cite{Su} for a proof in the $p$-adic case which is also valid in the finite fields case.

 We might ask the question: for which irreducible representation $\pi$ of $\SL_2(k)$, the induced representation $\Ind_{J}^{\Sp_4(k)}(\pi\otimes \omega_\psi)$ is multiplicity free?

For a $p$-adic field $k$, given an irreducible genuine irreducible representation $\pi$ of $\wt{\SL}_2(k)$, and an irreducible smooth representation $\sigma$ of $\Sp_4(k)$, it is always true that 
$$\dim\Hom_J(\sigma,\pi\otimes\omega_\psi)\le 1,$$
which is the main theorem of \cite{BR}. Any nonzero element in $\Hom_J(\sigma,\pi\otimes\omega_\psi)$ gives an embedding $\sigma\incl \Ind_J^{\Sp_4(k)}(\pi\otimes\omega_\psi)$, which is called a \textbf{Fourier-Jacobi} model of $\sigma$ for the given datum $(\pi\otimes\omega_\psi, J)$. Thus the above result of \cite{BR} says that the Fourier-Jacobi model of $\sigma$ is unique (if it exists). Fourier-Jacobi models over local fields were defined in a more general context for many classical groups such as $\Sp_{2n}$, unitary groups and $\GL_n$ in \cite{GGP}; the uniqueness of Fourier-Jacobi model in the $p$-adic field case was proved in \cite{GGP,Su}, and was proved in \cite{LS} in the Archimedean case.

For a finite field $k$, our purpose is to show that $\Ind_J^{\Sp_4(k)}(\pi\otimes\omega_\psi)$ is multiplicity free when $\pi$ is an irreducible representation not of the form $ I(\chi)$, where $\chi$ is a non-quadratic character of $k^\times$, see Section \ref{sec4}. When $\pi=I(\chi)$, it turns out that the induced representation $\Ind_J^{\Sp_4(k)}(I(\chi)\otimes\omega_\psi)$ is in general not multiplicity free, see Remark \ref{rmk4.2}. This in fact is not surprising after Proposition \ref{prop2.1}.

\subsection{The general strategy}\label{sec3.2} Given a  group $G$, an anti-involution ${}^\tau$ on $G$ is a map ${}^\tau:G\ra G$ such that ${}^\tau\!({}^\tau\! g)=g$ and ${}^\tau\! (g_1g_2)={}^\tau\! g_2{}^\tau\! g_1,$ for all $g,g_1,g_2\in G$.

In this subsection, let $G$ be an arbitrary finite group and $H$ be a subgroup of $G$.  Let $\sigma$ be a representation of $H$. We consider the algebra 
$$\CA(G,H,\sigma)=\wpair{K:G\ra \End_{\mathbb{C}}(\sigma): K(h_1gh_2)=\sigma(h_1)K(g)\sigma(h_2)}.$$
The product in $\CA(G,H,\sigma)$ is given by convolution:
$$K_1*K_2(g)=\sum_{x\in G}K_1(gx^{-1})K_2(x).$$
For $K\in \CA(G,H,\sigma)$ and $f\in \Ind_H^G(\sigma)$, we define a function $K*f:G\ra \sigma$ by 
$$(K*f)(x)=\frac{1}{|G|}\sum_{g\in G}K(xg^{-1})f(g).$$
One can check that $K*f\in \Ind_{H}^G(\sigma)$. Denote by $L_K\in \End(\Ind_H^G(\sigma))$ the endomorphism $f\mapsto K*f$.
\begin{thm}[Mackey, see {\cite[p.3]{Pr}}]\label{thm3.1}
The assignment $K\mapsto L_K$ defines an isomorphism
between $\CA(G,H,\sigma)$ and $\End_G(\Ind_H^G(\sigma))$.
\end{thm}
\begin{cor}
The induced representation $\Ind_H^G(\sigma)$ is multiplicity free if and only if the algebra $\CA(G,H,\sigma)$ is commutative. 
\end{cor}
\begin{proof}
By Schur's Lemma, the representation $\Ind_H^G(\sigma)$ is multiplicity free if and only if $\End_G(\Ind_H^G(\sigma))$ is commutative. Then the assertion follows from Mackey's Theorem, Theorem \ref{thm3.1}, directly.
\end{proof}
We are going to use the Gelfand-Kazhdan method to prove the commutativity of certain $\CA(G,H,\sigma)$.  We assume that there exists an anti-involution ${}^\tau$ of $G$ such that ${}^\tau H=H$, and there exists an anti-involution ${}^t$ on $\End_{\mathbb{C}}(\sigma)$ such that ${}^t (\sigma(h))=\sigma({}^\tau\! h)$ for all $h\in H$.  Then for $K\in \CA(G,H,\sigma)$, we can define ${}^\tau\! K:G\ra \End_{\mathbb{C}}(\sigma)$ by
$$({}^\tau \!K)(g)={}^t( K({}^\tau\! g)).$$
\begin{lem}\label{lem3.3}
For $K, K_1,K_2\in \CA(G,H,\sigma)$, we have
\begin{enumerate}
\item ${}^\tau \! K(h_1gh_2)=\sigma(h_1) {}^\tau\! K(g)\sigma(h_2), \forall h_1,h_2\in H,g\in G$; thus ${}^\tau \! K\in \CA(G,H,\sigma);$
\item ${}^\tau({}^\tau \! K)=K;$
\item ${}^\tau(K_1*K_2)={}^\tau \! K_2*{}^\tau \! K_1.$
\end{enumerate}
Thus, ${}^\tau$ is an anti-involution on $\CA(G,H,\sigma)$.
\end{lem}
The proof is routine and thus omitted. Certain detailed computation could be found in \cite{T}.
\begin{cor}\label{cor3.4}
Let $\wpair{g_i,1\le i\le n}$ be a subset of $G$ such that $G=\cup_{i,1\le i\le n}Hg_i H$. If  ${}^\tau \! K(g_i)=K(g_i)$ for all $i$ with $1\le i\le n$ for all $K\in \CA(G,H,\sigma)$, then $\CA(G,H,\sigma)$ is commutative and thus $\Ind_H^G(\sigma)$ is multiplicity free.
\end{cor}
\begin{proof}
Since $G=\cup Hg_iH$ and then ${}^\tau \! K(g)=K(g)$ for all $g\in G,K\in \CA(G,H,\sigma)$ from the assumption. Thus ${}^\tau \! K=K$. Since ${}^\tau$ is an anti-involution, we get that $\CA(G,H,\sigma)$ is commutative.
\end{proof}

\subsection{An anti-involution}
 Denote $d=\diag(-1,-1,1,1).$ For $g\in \Sp_4(k)$, we define
$${}^\iota g=d^{-1}gd. \quad {}^\tau\!g={}^\iota g^{-1}.$$
Then ${}^\iota$ is the MVW involution on $\Sp_4(k)$, see \cite[p.91]{MVW}, and ${}^\tau$ is an anti-involution on $\Sp_4(k)$.  We have 
$${}^\iota [x,y,z]=[x,-y,-z],$$
and $${}^\iota \begin{pmatrix}1&&&\\ &a&b& \\ &c&d&\\ &&&1 \end{pmatrix}=\begin{pmatrix}1&&&\\ &a&-b& \\ &-c&d&\\ &&&1 \end{pmatrix}.$$
In particular, ${}^\iota \! J=J$. Note that the restriction of the involution ${}^\iota$ to $\SL_2(k)$ is exactly the MVW involution considered in Section \ref{sec1.6}.
Given a representation $\sigma$ of $J$, we denote by ${}^\iota\! \sigma$ the representation  ${}^\iota \! \sigma(g)=\sigma({}^\iota\! g)$. 

Let $\pi$ be an irreducible representation of $\SL_2(k)$ which is not fully induced from the Borel subgroup and let $\sigma_\pi=\pi\otimes\omega_\psi$. The aim of the rest of this section is to define a transpose ${}^t$ on $\End(\sigma_\pi)$ such that ${}^t (\sigma(j))=\sigma({}^\tau \! j)$ for all $j\in J$.
First we need to define pairs between $\omega_\psi$ and ${}^\iota\!\omega_\psi$, $\pi$ and ${}^{\iota}\pi$.

\subsection{A pair on \texorpdfstring{$\omega_\psi\times {}^\iota\!\omega_\psi$}{Lg}}\label{sec3.4} 

For $a,b\in k$, recall the delta function $$\delta_{a,b}=\left\{\begin{array}{ll}1, & \textrm{ if } a=b;\\ 0, & \textrm{ if } a\ne b.\end{array} \right.$$
Then the space $\CS(k)$ has a basis $\wpair{\delta_s,s\in k}$, where $\delta_s(t)=\delta_{s,t}$.

\begin{lem}\label{lem3.5}
Consider the pair
$$\pair{\phi,\phi'}=\sum_{\xi\in k}\phi(\xi)\phi'(\xi),\phi,\phi'\in \CS(k).$$
We have $$\pair{\omega_\psi(j)\phi, \omega_\psi({}^\iota\!j)\phi'}=\pair{\phi,\phi'},\forall j\in J,\phi,\phi'\in \CS(k).$$
\end{lem}
\begin{proof}
 If $j=[x,0,z]$, we have $\omega_\psi(j)\phi(\xi)=\psi(z)\phi(\xi+x)$. On the other hand, we have ${}^\iota[x,0,z]=[x,0,-z]$. Thus ${}^\iota \omega_\psi(j)\phi'(\xi)=\omega_\psi([x,0,-z])\phi'(\xi)=\psi(-z)\phi'(\xi+x)$. By changing variable on the summation, we get
$$\pair{\omega_\psi([x,0,z]\phi),{}^\iota \omega_\psi([x,0,z])\phi'}=\sum_{\xi\in k}\phi(\xi)\phi'(\xi)=\pair{\phi,\phi'}.$$
Similarly, we can show that 
$$\pair{\omega_\psi(j)\phi,{}^\iota \omega_\psi(j)\phi'}=\pair{\phi,\phi'},\textrm{ for }j=[0,y,0], \diag(a,a^{-1}), \bx_\beta(b).$$

Let $w=\begin{pmatrix}&1\\ -1 \end{pmatrix}$ and $w'=\begin{pmatrix}&-1\\ 1 \end{pmatrix}$, then under the embedding $\SL_2(k)\incl \Sp_4(k)$, we have ${}^\iota w=w'.$ We need to show that
$$\pair{\omega_\psi(w)\phi,\omega_\psi(w')\phi'}=\pair{\phi,\phi'}.$$
Since $\CS(k)$ is spanned by $\wpair{\delta_s}$ and the pair $\pair{~,~}$ is bilinear, it suffices to show that 
$$\pair{\omega_\psi(w)\delta_s,\omega_\psi(w')\delta_t}=\pair{\delta_s,\delta_t}=\delta_{s,t}.$$
By the formula (\ref{eq2.1}), we have 
$$\omega_\psi(w)\delta_s(\xi)=\frac{1}{\gamma(1,\psi)}\psi(2s\xi),\omega_\psi(w')\delta_t(\xi)=\frac{1}{\gamma(-1,\psi)}\psi(-2t\xi).$$
Thus we get 
$$\pair{\omega_\psi(w)\delta_s,\omega_\psi(w')\delta_t}=\frac{1}{\gamma(1,\psi)\gamma(-1,\psi)}\sum_{\xi\in k}\psi(2(s-t)\xi)=\frac{q}{\gamma(1,\psi)\gamma(-1,\psi)}\delta_{s,t},$$
where the last step follows from $\sum_{\xi\in k}\psi(2(s-t)\xi)=q\delta_{s,t}. $ On the other hand, it is well-known that $\gamma(1,\psi)\gamma(-1,\psi)=q$, see \cite[Exercise 4.1.14, p.420]{Bu} for example. Thus we get 
$$\pair{\omega_\psi(w)\delta_s,\omega_\psi(w')\delta_t}=\pair{\delta_s,\delta_t}=\delta_{s,t}. $$

Since $J=\SL_2(k)\ltimes \sH$ is generated by $\diag(a,a^{-1}), \bx_\beta(b)$, $w$, $[x,0,z]$ and $[0,y,0]$, we get that 
$$\pair{\omega_\psi(j)\phi,{}^\iota \omega_\psi(j)\phi'}=\pair{\phi,\phi'}, \forall j\in J.$$
This completes the proof of the lemma.
\end{proof}
Note that the pair constructed in Lemma \ref{lem3.5} is symmetric and satisfies the property
\begin{equation}\label{eq3.1}
\pair{\delta_s,\delta_t}=\delta_{s,t}.
\end{equation}

\subsection{A pair on \texorpdfstring{$\St\otimes {}^\iota \St$}{Lg}}\label{sec3.5} Let $1$ be the trivial character of $k^\times$. We consider the induced representation $I(1)$ of $\SL_2(k)$. An element $f\in I(1)$ is a function $f:\SL_2(k)\ra \BC$ such that 
$$f\left(bg \right)=f(g),\forall b\in B_{\SL_2},g\in \SL_2(k),$$
where $B_{\SL_2}$ is the upper triangular subgroup of $\SL_2(k)$. 
For $f_1,f_2\in I(1)$, we define a pair 
$$\pair{f_1,f_2}=\sum_{g\in B_{\SL_2}\backslash \SL_2(k)}f_1(g)f_2(d_1gd_1),$$
where $d_1=\begin{pmatrix}-1&\\ &1 \end{pmatrix}\in \GL_2(k)$.
Note that this pair is well-defined and symmetric. Moreover, it satisfies the property 
\begin{align}
\pair{r(g)f_1,r({}^\iota\! g)f_2}=\pair{f_1,f_2}, \forall g\in \SL_2(k),f_1,f_2\in I(1),
\end{align}
where $r(g)f$ denotes the right translation action of $g$ on the sections $f$.

Recall that we have the Bruhat decomposition $\SL_2(k)=B_{\SL_2}\cup B_{\SL_2}wN_{\SL_2}$, where $w=\begin{pmatrix}&1\\ -1& \end{pmatrix}$. Let $f_0\in I(1)$ whose support is in $B_{\SL_2}$ and $f_0(b)=1,$ for all $b\in B_{\SL_2}$. For $r\in k$, let $f_{w,r}\in I(1)$ be the function such that its support is $B_{\SL_2}w \begin{pmatrix}1&r\\ &1 \end{pmatrix}$, and $f_{w,r}\left(bw  \begin{pmatrix}1&r\\ &1 \end{pmatrix}\right)=1, \forall b\in B_{\SL_2}$. Then $\wpair{f_0, f_{w,r}: r\in k}$ forms a basis of $I(1)$. Let $f_1=f_0+\sum_{r\in k}f_{w,r}$, i.e., $f_1(g)=1$ for all $g\in \SL_2(k)$.
\begin{lem} \label{lem3.6}
We have following formulas:
\begin{align}
\begin{split}\label{eq3.3}
\pair{f_1,f_1}&=q+1;\\
\pair{f_1,f_{w,r}}&=1,\forall r\in k;\\
\pair{f_{w,r},f_{w,s}}&=\delta_{s,-r}, \forall r,s\in k.
\end{split}
\end{align}
\end{lem}
\begin{proof}
This follows from direct computations.
\end{proof}
Recall that $I(1)=1\oplus \St$ as a representation of $\SL_2(k)$.  The subspace of $I(1)$ spanned by $f_1$ is an invariant subspace of $I(1)$ and it corresponds to the trivial representation of $\SL_2(k)$. For $r\in k$, define $F_r=f_1-f_{w,r}$. One can check that the space generated by $\{F_r,r\in k\}$ is also invariant under the action of $\SL_2(k)$, which is the space of $\St$. By Lemma \ref{lem3.6}, we have 
\begin{equation}\label{eq3.4}
\pair{F_r,F_{r'}}=q-1+\delta_{r,-r'}.
\end{equation}

\subsection{A pair on \texorpdfstring{$\omega_\psi^{\pm}\otimes {}^\iota\omega_\psi^{\pm}$}{Lg}}\label{sec3.6}
Recall that the space $\CS^+(k)$ of $\omega_\psi^+$ consists of functions $\phi\in \CS(k)$ with $\phi(-x)=\phi(x)$ for all $x\in k$. 
Let $A_0$ be a set of representatives of $k^\times/\wpair{\pm 1}$.
Then the set $\wpair{2\delta_0,\delta_s+\delta_{-s},s\in A_0}$ forms a basis of $\CS^+(k)$. For simplicity we write $\Delta_s=\delta_s+\delta_{-s}$ for $s\in k^\times$ and $\Delta_0=2\delta_0$.  The pair $\pair{~}$ defined in Lemma \ref{lem3.5} gives a pair on $\omega_\psi^+$ which satisfies
 $$\pair{\omega_\psi^+(g)\phi, \omega_\psi^+({}^\iota\! g)\phi'}=\pair{\phi,\phi'},\forall g\in \SL_2(k), \phi,\phi'\in \CS^+(k).$$
We have the formula
\begin{equation} \label{eq3.5}
\pair{\Delta_s,\Delta_t}=2(\delta_{s,t}+\delta_{s,-t}).
\end{equation}

Recall that the space of $\omega_\psi^-$ consists of  $\phi\in \CS(k)$ such that $\phi(-x)=-\phi(x)$ for all $x\in k$. For $s\in k^\times$, write  $\Delta'_s=\delta_s-\delta_{-s}$. Then $\wpair{\Delta'_s,s\in A_0}$ forms a basis of $\omega_\psi^-$. The pair $\pair{~}$ defined in Lemma \ref{lem3.5} also gives a bilinear symmetric pair on $\omega_\psi^-$ which satisfies 
$$\pair{\omega_\psi^-(g)\phi,\omega_\psi^-({}^\iota g)\phi'}=\pair{\phi,\phi'}, \forall g\in \SL_2(k),\phi,\phi'\in \omega_\psi^-.$$
And we also have a formula
\begin{equation}\label{eq3.6}\pair{\Delta_s',\Delta_t'}=2(\delta_{s,t}-\delta_{s,-t}).\end{equation}

\subsection{A pair on \texorpdfstring{$\omega_{\psi,\mu}\otimes {}^\iota\omega_{\psi,\mu}$}{Lg}}\label{sec3.7}
 Let $\mu$ be a character of $E^1$ with $\mu^2\ne 1$. We then have an irreducible cuspidal representation $\omega_{\psi,\mu}$ of $\SL_2(k)$. We need an explicit pair $\pair{~}:\omega_{\psi,\mu}\times \omega_{\psi,\mu}\ra \BC$ such that 
$$\pair{\omega_{\psi,\mu}(g)v_1,\omega_{\psi,\mu}({}^\iota \! g)v_2}=\pair{v_1,v_2}, \forall g\in \SL_2(k),v_1,v_2\in \omega_{\psi,\mu}.$$
Recall that the space of $\omega_{\psi,\mu}$ consists of functions $f:E^\times \ra \BC$ such that $f(yx)=\mu^{-1}(y)f(x)$ for all $y\in E^1,x\in E^\times$. 
As in Section \ref{cuspidal representations}, 
for each $a\in k^\times$, we fix an element $x_a\in E^\times$ such that $\Nm(x_a)=a$, then a function $f\in \omega_{\psi,\mu}$ is uniquely determined by its values on the set $\wpair{x_a,a\in k^\times}$. For each $a\in k^\times$, we define a function $f_a\in \omega_{\psi,\mu}$ such that $f_a(x_b)=\delta_{a,b}$. Then $\wpair{f_a:a\in k^\times} $ forms a basis of $\omega_{\psi,\mu}$. 

For $\phi,\phi'\in \omega_{\psi,\mu}$, notice that the function $x\mapsto \phi(x)\phi'(x^{q})$ on $E^\times$ is $E^1$-invariant. We define a pair $$\pair{\phi,\phi'}=\sum_{x\in E^1\backslash E^\times}\phi(x)\phi'(x^{q})=\frac{1}{q+1}\sum_{x\in E}\phi(x)\phi'(x^q),$$
where we used $\phi(0)=0$ for $\phi\in \omega_{\psi,\mu}$.
Then, we have 
\begin{equation}\label{eq3.7}
\pair{f_a,f_b}=\mu^{-1}(x_a^{q-1})\delta_{a,b},\forall a,b\in k^\times.
\end{equation}
\begin{lem}
We have $\pair{\omega_{\psi,\mu}(g)\phi,\omega_{\psi,\mu}({}^\iota\! g)\phi'}=\pair{\phi,\phi'}, \forall \phi,\phi'\in \omega_{\psi,\mu},g\in \SL_2(k).$
\end{lem}
\begin{proof}
Note that $a^q=a$ for $a\in k^\times$, and $\Nm(\xi)=\Nm(\xi^q)=\xi^{q+1}$ for $\xi\in E$. By Eq.(\ref{eq2.2}) and a simple changing of variables, we get $$\pair{\omega_{\psi,\mu}\left(\begin{pmatrix}a&\\ &a^{-1} \end{pmatrix}\right)\phi,\omega_{\psi,\mu}\left(\begin{pmatrix}a&\\ &a^{-1}\end{pmatrix} \right)\phi'}=\pair{\phi,\phi'},$$
and 
$$\pair{\omega_{\psi,\mu}\left(\begin{pmatrix}1&b\\ &1 \end{pmatrix}\right)\phi,\omega_{\psi,\mu}\left(\begin{pmatrix}1&-b\\ &1 \end{pmatrix} \right)\phi}=\pair{\phi,\phi'},$$
for all $a\in k^\times,b\in k, \phi,\phi'\in \omega_{\psi,\mu}$. 

We now check
$$\pair{\omega_{\psi,\mu}(w)\phi,\omega_{\psi,\mu}({}^\iota \! w)\phi'}=\pair{\phi,\phi'}.$$
It suffices to show that 
$$\pair{\omega_{\psi,\mu}(w)f_a,\omega_{\psi,\mu}({}^\iota \! w)f_b}=\mu^{-1}(x_a^{q-1})\delta_{a,b},\forall a,b\in k^\times. $$
We have 
$$\omega_{\psi,\mu}(w)f_a(\xi)=-q^{-1}\sum_{x\in E^1}\psi(\Tr(x^qx_a^q\xi))\mu^{-1}(x),$$
and 
$$\omega_{\psi,\mu}({}^\iota\!w)f_b(\xi)=-q^{-1}\sum_{y\in E^1}\psi(-\Tr(y^qx_b^q\xi))\mu^{-1}(y).$$
Thus 
\begin{align*}
&\ \pair{\omega_{\psi,\mu}(w)f_a,\omega_{\psi,\mu}({}^\iota \! w)f_b}\\
=&\ q^{-2}\frac{1}{q+1}\sum_{\xi\in E}\sum_{x,y\in E^1}\psi(\Tr(x^qx_a^q\xi)-\Tr(y^qx_b^q\xi^q))\mu^{-1}(xy)\\
=&\ \frac{q^{-2}}{q+1}\sum_{\xi\in E}\sum_{x,y\in E^1}\psi(\Tr(x^{-1}y^{-1}x_a^q\xi)-\Tr(x_b^q\xi^q))\mu^{-1}(xy), \quad \xi\mapsto \xi y^{-1}\\
=&\ q^{-2}\sum_{\xi\in E}\sum_{x\in E^1}\psi(\Tr(x^{-1}x_a^q\xi)-\Tr(x_b^q\xi^q))\mu^{-1}(x), \quad x\mapsto xy^{-1}\\
=&\ q^{-2}\sum_{\xi\in E}\sum_{x\in E^1}\psi(\Tr(x^{-1}x_a^q\xi)-\Tr(x_b\xi))\mu^{-1}(x), \quad \textrm{since } \Tr(x_b^q\xi^q)=\Tr(x_b\xi)\\
=&\ q^{-2}\mu^{-1}(x_a^{q-1})\sum_{\xi\in E}\sum_{x\in E^1}\psi(\Tr(x^{-1}x_a\xi)-\Tr(x_b\xi))\mu^{-1}(x), \quad x\mapsto xx_a^{q-1}.
\end{align*}
If $a\ne b$, for any $x\in E^1$, the character $\xi\mapsto \psi(\Tr(x^{-1}x_a\xi)-\Tr(x_b\xi))$ on $E$ is non-trivial, and thus $ \pair{\omega_{\psi,\mu}(w)f_a,\omega_{\psi,\mu}({}^\iota \! w)f_b}=0$. If $a=b$, the character $\xi\mapsto \psi(\Tr(x^{-1}x_a\xi)-\Tr(x_b\xi))$ on $E$ is non-trivial unless $x=1$, and thus $\pair{\omega_{\psi,\mu}(w)f_a,\omega_{\psi,\mu}({}^\iota \! w)f_b}=\mu^{-1}(x_a^{q-1})$. This completes the proof. 
\end{proof}

\subsection{Transpose operators on \texorpdfstring{$\End(\pi\otimes \omega_\psi)$}{Lg}}\label{sec3.8}
Let $\pi=1, \St,\omega_{\psi}^{\pm},\omega_{\psi_\kappa}^{\pm},$ or $\omega_{\psi,\mu}$, and $\sigma_\pi=\pi\otimes\omega_\psi$. We have constructed pairs on $\pi\times{}^\iota\pi$ and on $\omega_\psi\times{}^\iota\omega_\psi$ in previous subsections. We then can define a pair on $\sigma_\pi$ by 
$$\pair{f_1\otimes\phi_1,f_2\otimes\phi_2}=\pair{f_1,f_2}\pair{\phi_1,\phi_2}, f_1,f_2\in \pi,\phi_1,\phi_2\in \omega_\psi.$$
From the construction, we have 
$$\pair{\sigma_\pi(j)\Phi,{}^\iota\sigma_\pi(j)\Phi'}=\pair{\Phi,\Phi'},\forall j\in J, \Phi,\Phi'\in \sigma_\pi.$$

For $A\in \End_{\mathbb{C}}(\sigma_\pi)$, we define ${}^t\!A\in \End_{\mathbb{C}}(\sigma_\pi)$ by 
$$\pair{{}^t\!A(\Phi),\Phi'}=\pair{\Phi,A(\Phi')},\forall \Phi,\Phi'\in \sigma_\pi.$$
\begin{lem}
The assignment $A\mapsto {}^t \! A$ is an anti-involution on $\End_{\mathbb{C}}(\sigma_\pi)$ and satisfies
$${}^t(\sigma_\pi(j))=\sigma_\pi({}^\tau \!j),\forall j\in J.$$
\end{lem}
\begin{proof}
Note that the pair $\pair{~,~}$ on $\sigma_\pi$ is in fact symmetric from the construction, and it is routine to check that $A\mapsto {}^t\!A$ is an anti-involution on $\End_{\mathbb{C}}(\sigma_\pi)$. 

For $\Phi,\Phi'\in \sigma_\pi,j\in J$, we have that 
$$\pair{{}^t(\sigma_\pi(j))\Phi,\Phi'}=\pair{\Phi,\sigma_\pi(j)\Phi'}=\pair{\sigma_\pi({}^\tau\!j)\Phi,\Phi'}.$$
Thus we get ${}^t(\sigma_\pi(j))=\sigma_\pi({}^\tau \!j). $
\end{proof}

Let $\sigma_\pi|_{\SL_2(k)}=\bigoplus_i V_i$ be the decomposition given in Proposition \ref{prop2.1}, where $V_i$ is an irreducible representation of $\SL_2(k)$. Note that each $V_i$ occurred at most once in the decomposition. Let $\id_{V_i}\in \End(\sigma_\pi)$ be the element such that $\id_{V_i}|_{V_j}=0$ if $j\ne i$, and $\id_{V_i}|_{V_i}$ is the identity.
\begin{lem}\label{lem3.9}
We have ${}^t\id_{V_i}=\id_{V_i}$.
\end{lem}
\begin{proof}
Suppose that $j\ne i$, we need to show that ${}^t \id_{V_i}|_{V_j}=0$. Suppose that this is false, then there exists $v_j\in V_j$ such that $ {}^t\id_{V_i}(v_j)\ne 0$. To get a contradiction, it suffices to show that $\pair{{}^t\id_{V_i}(v_j),v}=0$ for all $v\in \sigma_\pi$. If $v\notin V_i$, we have $\pair{{}^t\id_{V_i}(v_j),v}=\pair{v_j,\id_{V_i}v}=0 $ since  $\id_{V_i}v=0$. If $v\in V_i$, we have $\pair{{}^t\id_{V_i}(v_j),v}=\pair{v_j,v}$. If this is not zero, we then get a non-trivial pair between $V_i$ and $V_j$ such that $\pair{\sigma_\pi(g)v,\sigma_\pi({}^\iota\!g)v'}\ne 0$, which would imply that $\wt V_i={}^\iota V_j$. But we know that $\wt V_i={}^\iota V_i$ by Lemma \ref{lem2.5} and ${}^\iota V_j $ is not isomorphic to ${}^\iota V_i$ by assumption. This proves the lemma. 
\end{proof}

\section{Certain multiplicity one theorems for \texorpdfstring{$\Sp_4(k)$}{Lg}}\label{sec4}

 Our main theorem for $\Sp_4(k)$ is the following

\begin{thm}\label{thm4.1}
The representation $\Ind_J^{\Sp_4(k)}(\pi\otimes\omega_\psi)$ of $\Sp_4(k)$ is multiplicity free if $\pi=1,\St,\omega_{\psi}^{\pm},\omega_{\psi_\kappa}^{\pm}$, $\omega_{\psi,\mu}$, where $\mu$ is a character of $E^1$ with $\mu^2\ne 1$.
\end{thm}

\begin{rmk}\label{rmk4.2}
\rm{
Before proving Theorem \ref{thm4.1}, we show that for $q$ large, the representation $\Ind_J^{\Sp_4}(I(\chi)\otimes \omega_\psi)$ is not multiplicity free. In the following, we write $\Sp_4(k)$ as $G$ for simplicity. Recall that $P$ is the Kilingen parabolic subgroup with Levi $M\cong \GL_1(k)\times \SL_2(k)$. Given characters $\chi_1,\chi_2$ of $k^\times$, view $\chi_1\otimes I(\chi_2)$ as a representation of $M\cong \GL_1(k)\times \SL_2(k)$ and consider the parabolic induction $\Ind_P^G(\chi_1\otimes I(\chi_2))$. We claim that, if $\chi_2=\epsilon \chi$, then $$\Hom_G(\Ind_J^G(I(\chi)\otimes \omega_\psi),\Ind_P^G(\chi_1\otimes I(\chi_2)))\ge 2.$$
In fact, by Frobenius reciprocity law,  
\begin{align*}
    \Hom_G(\Ind_J^G(I(\chi)\otimes \omega_\psi),\Ind_P^G(\chi_1\otimes I(\chi_2)))&=\Hom_J(I(\chi)\otimes \omega_\psi,\Ind_P^G(\chi_1\otimes I(\chi_2))|_J).
\end{align*}
By Mackey's Theorem (see \cite[Proposition 22, p.58]{Se}), 
$$\Ind_P^G(\chi_1\otimes I(\chi_2))|_J=\bigoplus_{s\in J\backslash G/P}\Ind_{P_s}^J((\chi_1\otimes I(\chi_s))^s), $$
where $P_s=sPs^{-1}\cap J$, and for a representation $\rho$ of $P$, the representation $\rho^s$ of $P_s$ is defined by $\rho^s(h)=\rho(s^{-1}hs)$. Considering the element $s=s_\alpha s_\beta s_\alpha \in J\backslash G/P$, we have that 
\begin{align*}
    \Hom_G(\Ind_J^G(I(\chi)\otimes \omega_\psi),\Ind_P^G(\chi_1\otimes I(\chi_2)))&\supset\Hom_J(I(\chi)\otimes \omega_\psi,\Ind_{P_s}^J(\chi_1\otimes I(\chi_2))^s)\\
    &=\Hom_{P_s}(I(\chi)\otimes \omega_\psi|_{P_s}, (\chi_1\otimes I(\chi_2))^s).
\end{align*}
We have $P_s\cong \SL_2(k)\incl M$, and $(\chi_1\otimes I(\chi_2))^s =I(\chi_2)$. Thus 
$$\dim \Hom_G(\Ind_J^G(I(\chi)\otimes \omega_\psi),\Ind_P^G(\chi_1\otimes I(\chi_2)))\ge \dim\Hom_{\SL_2(k)}(I(\chi)\otimes \omega_\psi,I(\chi_2)). $$
By Proposition $\ref{prop2.1}$, if $\chi_2=\epsilon \chi$, then
$$ \dim\Hom_{\SL_2(k)}(I(\chi)\otimes \omega_\psi,I(\epsilon \chi))=2.$$
Thus $$\dim \Hom_G(\Ind_J^G(I(\chi)\otimes \omega_\psi),\Ind_P^G(\chi_1\otimes I(\epsilon\chi)))\ge 2.$$

By Mackey's irreducibility criterion (see \cite[p.59]{Se}), if $\chi_1$ and $\epsilon\chi$ are in ``general position", the induced representation $\Ind_P^G(\chi_1\otimes I(\epsilon \chi))$ is irreducible. Here two characters $\chi_1,\chi_2$ are said to be in general position, if $(\chi_1\otimes \chi_2)\ne (\chi_1\otimes\chi_2)^w$ for all $w\in W(\Sp_4)-\wpair{1}$, where $W(\Sp_4)$ denotes the Weyl group of $\Sp_4$ and $(\chi_1\otimes \chi_2)$ is viewed as a character of the maximal torus of $G$ via
$$(\chi_1\otimes\chi_2)(\diag(a,b,b^{-1},a^{-1}))=\chi_1(a)\chi_2(b),$$
and $(\chi_1\otimes\chi_2)^w(t)=(\chi_1\otimes \chi_2)(w.t)$ for $t$ in the maximal torus. In fact, it is not hard to check that $\chi_1,\chi_2$ are in general position if and only if $\chi_1^2\ne 1,\chi_2^2\ne 1,\chi_1\ne \chi_2^{\pm1}$. For $q$ large (in fact, $q\ge 7$ will suffice), one could find $\chi_1$ such that $\chi_1,\epsilon\chi$ are in general position, so that $\Ind_P^G(\chi_1\otimes I(\epsilon \chi))$ is irreducible. Hence $\Ind_J^G(I(\chi)\otimes\omega_\psi)$ is not multiplicity free for $q$ large. \qed}
\end{rmk}

Before we start the proof of Theorem \ref{thm4.1}, we also need to give the double coset decomposition $J\backslash \Sp_4(k)/J$.
Denote $t(a)=\diag(a,1,1,a^{-1})$ for $a\in k^\times$. 
 From the decomposition 
$$\Sp_4(k)=Ps_\alpha s_\beta s_\alpha P\cup Ps_\alpha P\cup P,$$
and $$P=\cup _{a\in k^\times}t(a)J=\cup_{a\in F^\times}Jt(a),$$
we can get a set of representatives of the double coset $J\backslash \Sp_4(k)/J$ given by 
$$t(a), \eta(a):= \begin{pmatrix}&a&&\\ a^{-1} &&&\\ &&&a\\ &&a^{-1}& \end{pmatrix}, \xi(a):= \begin{pmatrix}&&a\\ &I_2&\\ -a^{-1}&&\end{pmatrix}, a\in k^\times.$$

\noindent\textbf{Proof of Theorem $\ref{thm4.1}$.}
If $\pi=1$, the multiplicity-freeness of $\Ind_J^{\Sp_4(k)}(\omega_\psi)$ could be deduced from the main result of \cite{T}. In the following, for completeness, we still give details of the proof in this case.

 Denote $\sigma_\pi=\pi\otimes\omega_\psi$  for $\pi$ listed in Theorem \ref{thm4.1}. In Section \ref{sec3.8}, we have constructed an anti-involution ${}^t$ on $\End(\sigma_\pi)$ such that ${}^t(\sigma_\pi(j))=\sigma_\pi({}^\tau\! j)$.
We can define an anti-involution ${}^\tau$ on $\CA(\Sp_4(k),J,\sigma_\pi)$ by 
$$({}^\tau\! K)(g)={}^t\! (K({}^\tau \! g)), K\in \CA(\Sp_4(k),J,\sigma_\pi), g\in \Sp_4(k).$$

By Corollary \ref{cor3.4}, it suffices to show that $({}^\tau\! K)(g)=K(g)$ for $g=t(a),\eta(a),\xi(a)$ for all $a\in k^\times$ and all $K\in \CA(\Sp_4(k),J,\sigma_\pi)$. Replacing $K$ by $K-{}^\tau \! K$, it suffices to show that for $K\in \CA(G,H,\sigma_\pi)$ with ${}^\tau \! K=-K$, $K(g)=0$ for $g=t(a),\eta(a),\xi(a),\forall a\in k^\times$. We shall assume ${}^\tau \! K=-K$ and 
show that $K(g)=0$ for $g=t(a),\eta(a),\xi(a),a\in k^\times$, case by case.

Step (1), we show that $K(t(a))=0$ for all $a\in k^\times$. We first consider $t(a), a\ne \pm 1$. 
Since 
$t(a)[0,0,z]=[0,0,a^2z]t(a),$  
from the definition of $\CA(\Sp_4(k), J,\sigma_\pi)$,  $\psi(z)K(t(a))=\psi(a^2z)K(t(a))$. Since $\psi$ is non-trivial and $a^2\ne 1$, one can choose $z\in k$ such that $\psi(z)\ne \psi(a^2z).$ Hence, $K(t(a))=0.$ 
Next, we show that $K(t(a))=0$ if $a^2=1$. Since $t(a)g=gt(a),\forall g\in \SL_2(k)$,  $$K(t(a))\sigma_\pi(g)=\sigma_\pi(g)K(t(a)),\forall g\in \SL_2(k).$$
This implies that $K(t(a))\in \End_{\SL_2(k)}(\sigma_\pi)$. As a representation of $\SL_2(k)$, by Proposition \ref{prop2.1}, we can write
$$\sigma_\pi=\bigoplus V_i,$$
where $V_i$ is an irreducible representation of $\SL_2(k)$ and $i$ runs in certain index set. By Schur's Lemma, we can write 
$$K(t(a))=\sum_i C_i \id_{V_i},$$
with $C_i\in \BC$ depending on $a$. On the other hand, we have ${}^\tau(t(a))=t(a)$ if $a^2=1$. Thus $({}^\tau \! K)(t(a))={}^t(K(t(a)))$ by definition. By Lemma \ref{lem3.9}, the idempodents $\id_{V_i}$ are invariant under transpose, which implies that ${}^t(K(t(a)))=K(t(a))$.  Then the assumption ${}^\tau \! K=-K$ implies that $K(t(a))=0.$ This completes Step (1).

Step (2), we show that $K(\eta(a))=0$ for all $a\in k^\times$. We first record the following relations
\begin{align}
    \eta(a)\bx_\beta(y)&=[0,0,a^2y]\eta(a), \label{eq4.1}\\
    \eta(a)[0,0,a^2y]&=\bx_\beta(y)\eta(a), \label{eq4.2}\\
    \eta(a)[0,y,0]&=[0,y,0]\eta(a).\label{eq4.3}
\end{align}

We now consider the cases $\pi=1, \St, \omega_\psi^{\pm},\omega_{\psi_\kappa}^{\pm},\omega_{\psi,\mu}$, respectively. 

Case (2.1), $\pi=1$. Recall that $\sigma_\pi=\omega_\psi$ has a basis $\wpair{\delta_s,s\in k}.$
Applying formulas (\ref{eq2.1}), we have that
\begin{equation}\label{eq4.4}\omega_\psi(\bx_\beta(y))\delta_s=\psi(ys^2)\delta_s,\quad \omega_\psi([0,y,0])\delta_s=\psi(2sy)\delta_s, \forall y\in k.\end{equation}
Assume that $K(\eta(a))\delta_s=\sum_{t\in k}C_s(t)\delta_t$, for $C_s(t)\in \BC$.
From the relation (\ref{eq4.3}), 
$$K(\eta(a))\sigma_\pi[0,y,0]=\sigma_\pi[0,y,0]K(\eta(a)).$$
Applying the above formula to $\delta_s$ we have that
$$\psi(2sy)\sum_t C_s(t)\delta_t=\sum_t \psi(2ty)C_s(t)\delta_t.$$
Hence, $C_s(t)=0$ if $t\ne s$, and $K(\eta(a))\delta_s=C_s(s)\delta_s$.

We now show that $C_s(s)=0$ using the assumption that $K+{}^\tau\!K=0$. Since ${}^\tau \! \eta(a)=\eta(a)$, ${}^\tau \! K(\eta(a))={}^t (K(\eta(a)))$. Thus we have ${}^t \! K(\eta(a))+K(\eta(a))=0$, which then implies that 
$$\pair{{}^t (K(\eta(a))\delta_s,\delta_s}+\pair{K(\eta(a))\delta_s, \delta_s}=0, \forall s\in k. $$
From the definition of ${}^t(K(\eta(a)))$, the above condition implies that $\pair{K(\eta(a))\delta_s,\delta_s}=0.$ From Eq.(\ref{eq3.1}), we then get $C_s(s)=0$. This shows that $K(\eta(a))=0$, for all $a\in k^\times$. 

Case (2.2), $\pi=\St$. Recall that $\sigma_\pi=\St\otimes\omega_\psi$ has a basis $\wpair{F_r\otimes \delta_s,r,s\in k},$ see Section \ref{sec3.5}. Note that 
\begin{equation}\label{eq4.5}\St(\bx_\beta(b))F_r=F_{r-b}, \quad \St([x,y,z])F_r=F_r,\end{equation}
where the action of $\St$ is given by right translation.

 From the relation Eq.(\ref{eq4.1}), we get $K(\eta(a))\sigma(\bx_\beta(r))=\psi(a^2y)K(\eta(a))$. Applying this formula to $F_r\otimes\delta_s$ and using Eq.(\ref{eq4.4}) (\ref{eq4.5}), we have
$$\psi(ys^2)K(\eta(a))F_{r-y}\otimes \delta_s=\psi(a^2y)K(\eta(a))F_r\otimes \delta_s.$$
In particular, $K(\eta(a))F_r\otimes \delta_s=\psi((s^2-a^2)r)K(\eta(a))F_0\otimes \delta_s.$ Thus to show that $K(\eta(a))=0$ it suffices to show that $K(\eta(a))F_0\otimes\delta_s=0,\forall s\in k.$

Assume that $K(\eta(a))F_0\otimes \delta_s=\sum_{b,t\in k}C_s(b,t)F_b\otimes \delta_t$. Here $C_s(b,t)\in \BC$ might also depend on $a$. From the relation \eqref{eq4.2}, we have $\psi(a^2y)K(\eta(a))=\sigma_\pi(\bx_\beta(y))K(\eta(a))$. Applying this to $F_0\otimes \delta_s$, we get that
\begin{align*}
\sum_{b,t}\psi(a^2y)C_s(b,t)F_b\otimes \delta_t&=\sum_{b,t}C_s(b,t)\psi(yt^2)F_{b-y}\otimes \delta_t\\
&=\sum_{b,t}C_s(b+y,t)\psi(yt^2)F_b\otimes \delta_t.
\end{align*}
Hence, 
\begin{equation}\label{eq4.6}
C_s(b+y,t)=\psi((a^2-t^2)y)C_s(b,t), \forall  b,t,y\in k.
\end{equation}
On the other hand, using the relation (\ref{eq4.3}), we have $K(\eta(a))\sigma([0,y,0])=\sigma([0,y,0])K(\eta(a))$. Applying this to $F_0\otimes \delta_s$ and using Eqs.(\ref{eq4.4}), (\ref{eq4.5}), we obtain that
\begin{align*}
\sum_{b,t\in k}\psi(2sy)C_s(b,t)F_b\otimes \delta_t=\sum_{b,t\in k}C_s(b,t)\psi(2ty)F_b\otimes \delta_t.
\end{align*}
Thus we get $\psi(2sy)C_s(b,t)=\psi(2ty)C_s(b,t)$ for all $b,t,y\in k$. If $s\ne t$, one can choose $y$ such that $\psi(2sy)\ne \psi(2ty)$, hence,  $C_s(b,t)=0$. Write $D_s=C_s(0,s)$, then  $C_s(b,s)=\psi((a^2-s^2)b)D_s$ by Eq.(\ref{eq4.6}). We get that 
\begin{align*}
K(\eta(a))F_0\otimes \delta_s&=\sum_b C_s(b,s)F_b\otimes \delta_s=\sum_b \psi((a^2-s^2)b)D_s F_b\otimes \delta_s.
\end{align*}
Since ${}^\tau\! \eta(a)=\eta(a)$, and ${}^\tau \! K=-K$, we get $\pair{K(\eta(a))F_0\otimes\delta_s,F_0\otimes\delta_s}=0,$
where the pair $\pair{~,~}$ is defined in Section \ref{sec3.8}. By Eqs.(\ref{eq3.1}) and (\ref{eq3.4}), the above equation is equivalent to 
$$D_s+(q-1)\sum_{b\in k}\psi((a^2-s^2)b)D_s=0.$$
Note that $ \sum_{b\in k}\psi((a^2-s^2)b)$ is either $q$ or $0$ depending on $a^2=s^2$ or not. We then have $D_s=0.$ This shows that $K(\eta(a))=0$ when $\pi=\St$.

Case (2.3), $\pi=\omega_{\psi_u}^{\pm 1}$ for $u=1,\kappa$. In these cases, the proofs are similar, and we only give details for $\pi=\omega_\psi^+$. Recall that $\sigma_\pi=\omega_{\psi}^+\otimes\omega_\psi$ has a basis $\Delta_r\otimes\delta_s$, where $r$ runs over $A_0\cup\wpair{0}$ and $s\in k$. Recall that $A_0$ is a set of representatives of $k^\times/\wpair{\pm 1}$ (see Section \ref{sec3.6}). 

We first record the following formulas
\begin{equation}\label{eq4.7}
\sigma([0,y,z])\Delta_r\otimes \delta_s=\psi(2sy)\Delta_r\otimes \delta_s, \quad
\sigma(\bx_\beta(b))\Delta_r\otimes \delta_s=\psi(b(r^2+s^2))\Delta_r\otimes \delta_s.
\end{equation}
From the relation $\eta(a)\bx_\beta(y)=[0,0,a^2y]\eta(a)$, we get $K(\eta(a))\sigma(\bx_\beta(y))=\psi(a^2y)K(\eta(a))$. Applying this to $\Delta_r\otimes\delta_s$, we get that
$$\psi((r^2+s^2)y)K(\eta(a))\Delta_r\otimes \delta_s=\psi(a^2y)K(\eta(a))\Delta_r\otimes \delta_s.$$
Since $y$ is arbitrary, we get that 
\begin{equation}\label{eq4.8}K(\eta(a))\Delta_r\otimes \delta_s=0 \textrm{ if } r^2+s^2\ne a^2\end{equation}

If $r^2+s^2=a^2$, assume that $K(\eta(a))\Delta_r\otimes \delta_s=\sum_{b,t}C_{r,s}(b,t)\Delta_b\otimes \delta_t$, where $t$ runs over $k$, and $b$ runs over $\wpair{0}\cup A_0$. 
 From the relations \eqref{eq4.2} and \eqref{eq4.7}, we can get $ \psi(a^2y)K(\eta(a))\Delta_r\otimes \delta_s=\sigma(\bx_\beta(y))K(\eta(a))\Delta_r\otimes \delta_s$, or 
$$\psi(a^2y)\sum_{b,t}C_{r,s}(b,t)\Delta_b\otimes \delta_t=\sum_{b,t}C_{r,s}(b,t)\psi(y(b^2+t^2))\Delta_b\otimes \delta_t. $$
Since $y$ is arbitrary, we can get 
\begin{equation}\label{eq4.9} C_{r,s}(b,t)=0, \textrm{ if } b^2+t^2\ne a^2.\end{equation}
From the relation \eqref{eq4.3}, we have $K(\eta(a))\sigma([0,y,0])=\sigma([0,y,0])K(\eta(a))$. By Eq.(\ref{eq4.7}), we have 
$$\psi(2sy)\sum_{b,t}C_{r,s}(b,t)\Delta_b\otimes \delta_t=\sum_{b,t}C_{r,s}(b,t)\psi(2ty)\Delta_b\otimes \delta_t.$$
Hence, $C_{r,s}(b,t)=0$ if $t\ne s$, and \begin{equation}\label{eq3.9}K(\eta(a))\Delta_r\otimes \delta_s=\sum_{b}C_{r,s}(b,s)\Delta_b\otimes \delta_s.\end{equation} Let $b_0\in k$ be such that $b_0^2=a^2-s^2$. If $b^2\ne b_0^2$, then $C_{r,s}(b,s)=0$ by Eq.(\ref{eq4.9}). Thus we get $K(\eta(a))\Delta_r\otimes \delta_s=C_{r,s}(b_0,s)\Delta_{b_0}\otimes \delta_s$. On the other hand, if $r\ne \pm b_0$, then $C_{r,s}(b_0,s)=0$ by Eq.(\ref{eq4.8}). Thus to show $K(\eta(a))=0$, it suffices to show that $K(\eta(a))\Delta_{b_0}\otimes \delta_s=0$. Note that  $K(\eta(a))\Delta_{b_0}\otimes \delta_s=C_{b_0,s}(b_0,s)\Delta_{b_0}\otimes \delta_s.$ Since ${}^\tau \! K=-K$ and ${}^\tau \!  \eta(a)=\eta(a)$, we have 
$$\pair{K(\eta(a))\Delta_{b_0}\otimes \delta_s,\Delta_{b_0}\otimes \delta_s}=0,$$
 By the definition of the $\pair{~,~}$ (see Section \ref{sec3.8}) and Eq.(\ref{eq3.5}), we get that 
$$0= \pair{K(\eta(a))\Delta_{b_0}\otimes \delta_s,\Delta_{b_0}\otimes \delta_s}=2C_{b_0,s}(b_0,s_0).$$
Thus we get $C_{b_0,s}(b_0,s)=0$.  This shows that $K(\eta(a))\Delta_{b_0}\otimes \delta_{s}=0$, and hence $K(\eta(a))=0$. Therefore, we get $K(\eta(a))=0$, for all $a \in k^{\times}$, when $\pi=\omega_\psi^+$.

Case (2.4), $\pi=\omega_{\psi,\mu}$ for a character $\mu$ of $E^1$ with $\mu^2\ne 1$. Recall that $\pi$ has a basis $\wpair{f_a,a\in k^\times}$, see Section \ref{sec3.7}. We record the following formulas
\begin{align}
\begin{split}\label{eq4.11}
\sigma([0,y,z])f_a\otimes \delta_s&=\psi(2sy+z)f_a\otimes \delta_s;\\
\sigma(\bx_\beta(b))f_a\otimes \delta_s&=\psi(b(a+s^2))f_a\otimes \delta_s.
\end{split}
\end{align}

From the relation \eqref{eq4.1}, we can get $K(\eta(a))\sigma(\bx_\beta(y))=\psi(a^2y)K(\eta(a)).$
Applying this to $f_r\otimes\delta_s$, we get $$\psi(y(r+s^2))K(\eta(y))f_r\otimes \delta_s=\psi(a^2y)f_r\otimes \delta_s.$$
Since $y$ is arbitrary, we get
$K(\eta(a))f_r\otimes \delta_s=0 \textrm{ if } r+s^2\ne a^2$. Assume that
$$K(\eta(a))f_{a^2-s^2}\otimes \delta_s=\sum_{b\in k^\times,t\in k}C_{s}(b,t)f_b\otimes \delta_t.$$
From the relations \eqref{eq4.2} and \eqref{eq4.11}, we can obtain that 
$$ \psi(a^2y)K(\eta(a))f_{a^2-s^2}\otimes \delta_s=\sigma(\bx_{\beta}(y))K(\eta(a))f_{a^2-s^2}\otimes \delta_s,$$
i.e., 
$$\psi(a^2y)\sum_{b,t}C_{s}(b,t)f_b\otimes \delta_t=\sum_{b,t}C_{s}(b,t)\psi(y(b+t^2))f_b\otimes \delta_t.$$
Thus we get 
\begin{equation}\label{eq4.12}C_{s}(b,t)=0 \textrm{ if } b+t^2\ne a^2.\end{equation} Furthermore, using \eqref{eq4.3}, we have $K(\eta(a))\sigma([0,y,0])f_{a^2-s^2}\otimes \delta_s=\sigma([0,y,0])K(\eta(a))f_{a^2-s^2}\otimes \delta_s$, which is equivalent to
$$\sum_{b,t}C_{s}(b,t)\psi(2sy)f_b\otimes \delta_t=\sum_{b,t}C_{s}(b,t)\psi(2ty)f_b\otimes \delta_t.$$
Thus we get
\begin{equation}\label{eq4.13}C_{s}(b,t)=0, \textrm{ if } t\ne s.\end{equation} 
Thus we get
$$K(\eta(a))f_{a^2-s^2}\otimes \delta_s=C_{s}(a^2-s^2,s)f_{a^2-s^2}\otimes \delta_s.$$
Since ${}^\tau\!K=-K$ and ${}^\tau \! \eta(a)=\eta(a)$, we get $\pair{K(\eta(a))f_{r_0}\otimes \delta_s,f_{r_0}\otimes \delta_s}=0$. By Eq.(\ref{eq3.7}), it is easy to see that $C_{s}(a^2-s^2,s)=0$. This shows that $K(\eta(a))=0$, for all $a \in k^{\times}$, in the case $\pi=\omega_{\psi,\mu}$. This also completes the proof of Step (2).

Step (3), we show that $K(\xi(a))=0$ for all $a\in k^\times$. 
One can check that ${}^\tau \! (\xi(a))=\xi(a)$ and $$\xi(a) g=g\xi(a),\forall g\in \SL_2(k).$$
Thus $$K(\xi(a))\sigma_\pi(g)=\sigma_\pi(g)K(\xi(a)), \forall g\in \SL_2(k).$$
Hence, $K(\xi(a))\in \End_{\SL_2(k)}(\sigma_\pi)$. Let $\sigma_\pi|_{\SL_2(k)}=\oplus V_i$ be the irreducible decomposition as in Proposition \ref{prop2.1}. As in the proof of $K(t(a))=0$ when $a^2=1$ in Step (1), we can write $K(\xi(a))=\sum_i C_i\id_{V_i}$, with $C_i\in \BC$ depending on $a$. Since $\id_{V_i}$ is invariant under the transpose $^t$ by Lemma \ref{lem3.9} and ${}^t \! K(\xi(a))+K(\xi(a))=0$, we can get $C_i=0$ and thus $K(\xi(a))=0.$

This completes the proof of Theorem \ref{thm4.1}.
\qed

\section{A multiplicity one theorem for \texorpdfstring{$\RU_4$}{Lg} over finite fields}
In this section, we briefly introduce a multiplicity one result for the unitary group $\RU_4(k)$, which is quite similar to the $\Sp_4(k)$ case. Note that some notations which were used for subgroups of $\Sp_4(k)$ in previous sections will be used for subgroups of $\RU_4$ in this subsection.

Recall that $k$ is a finite field with odd cardinality $q$ and $E$ is the quadratic extension of $k$.
Define
$$\RU_{2n}(k)=\wpair{g\in \GL_{2n}(E):  g \bpm &J_n\\-J_n& \epm {}^t\! \bar g=\bpm  &J_n\\ -J_n& \epm}.$$ 

\subsection{Conjugacy classes and some simple representations of \texorpdfstring{$\RU_2(k)$}{Lg}}
The conjugacy classes of $\RU_2(k)$ is given in the following table (see \cite{Ca}):

\begin{align*}
\begin{array}{|c|c|c|}
\hline
\textrm{Representative} & \textrm{Number of elements in class } & \textrm{ Number of classes }   \\
\hline
\begin{pmatrix}x &\\ &x \end{pmatrix},x\in E^1  &  1&q+1  \\
\hline
\begin{pmatrix}x&x\\ &x \end{pmatrix} ,x\in E^1 &(q-1)(q+1) & q+1   \\
\hline
 \begin{pmatrix}x &\\ &\bar x^{-1}\end{pmatrix},x\in E^\times-E^1&q(q+1)&\frac{(q+1)(q-2)}{2}\\
 \hline
\begin{pmatrix}x&y\\ \kappa y &x \end{pmatrix},y\ne 0&q(q-1)&\frac{q(q+1)}{2}\\
 \hline
\end{array}
\end{align*}
 Note that the norm map $\Nm:E^\times\ra k^\times$ is surjective and thus $\bpm 1&1\\ &1\epm$ and $\bpm 1&\kappa \\ &1\epm$ are in the same conjugacy classes. This is different from the $\SL_2(k)$ case. We explain a little bit about the last row. The condition $\begin{pmatrix}x&y\\ \kappa y &x \end{pmatrix}\in \RU_2(k)$ is equivalent to $\bar xy=x\bar y$ and $x\bar x-\kappa y\bar y=1$, which implies that $x\pm y\sqrt{\kappa}\in E^1$. Note that unlike in the $\SL_2(k)$ case, here we don't require that $x,y\in k$. We now count the number of representatives of the form $ \begin{pmatrix}x&y\\ \kappa y &x \end{pmatrix}$. If $x=0$, we get $y\sqrt{\kappa}\in E^1$ and there are totally $q+1$ such $y$. If $x\ne 0,y\ne 0$, let $u_1=x+y\sqrt{\kappa},u_2=x-y\sqrt{\kappa}$. Then $u_1,u_2\in E^1$,  $u_1\ne \pm u_2$. There are totally $(q+1)(q-1)$ choices of $u_1,u_2$, and hence such many of $x,y$. Note that $x+y\sqrt \kappa$ and $x-y\sqrt{\kappa }$ give the same conjugacy class. Thus we totally have $\frac{1}{2}(q+1+(q+1)(q-1))=q(q+1)/2$ classes in the last row. It is not hard to check that the number of elements in each class in the last row is $q(q-1)$ by counting the centralizer of each representative. Note that there are $(q+1)^2$ conjugacy classes and thus there are $(q+1)^2$ irreducible representations of $\RU_2(k)$.

Let $\eta$ be a character of $E^1$. View $\eta$ as a representation of $\RU_2(k)$ via the determinant map $\det:\RU_2(k)\ra E^1.$ We then have total $q+1$ irreducible 1-dimensional representations of $\RU_2(k)$.

Let $B_{\RU_2}=A_{\RU_2}\ltimes N_{\RU_2}$ be the upper triangular Borel subgroup of $\RU_2(k)$ with torus $A_{\RU_2}=\wpair{\diag(a,\bar a^{-1}),a\in E^\times}$ and unipotent subgroup $N_{\RU_2}$.
Given a character $\chi$ of $E^\times$, view it as a character on $A_{\RU_2}\cong E^\times$ and hence on $B_{\RU_2}$ such that the action of $N_{\RU_2}$ is trivial. We then consider the induced representation $I(\chi):=\Ind_{B_{\RU_2}}^{\RU_2(k)}(\chi)$ which is irreducible if and only if $\chi\ne \bar \chi^{-1}$, where $\bar \chi^{-1}$ is the character of $E^\times$ defined by $\bar \chi^{-1}(a)=\chi(\bar a^{-1}),a\in E^\times.$ Note that the condition $\chi=\bar \chi^{-1}$ is equivalent to $\chi\circ \Nm_{E/k}=1$. Since the norm map $\Nm_{E/k}:E^\times\ra k^\times$ is surjective, the condition $\chi=\bar \chi^{-1}$ is equivalent to that $\chi|_{k^\times}=1$ and giving such a character is amount to giving a character $\eta$ of $E^1$ via $\eta(a/\bar a)=\chi(a)$. On the other hand, we have $I(\chi)\cong I(\bar \chi^{-1})$ and thus there are totally $\frac{1}{2}(q+1)(q-2)$ irreducible representations of the form $I(\chi),\chi\ne \bar \chi^{-1}$. Note that $\dim I(\chi)=q+1.$

Consider the induced representation $I(1)$, where $1$ is the trivial character of $E^\times$. We have $I(1)=1\bigoplus \St$, where $1$ denotes the trivial representation of $\RU_2(k)$ by abuse of notation, $\St$ is the Steinberg representation and $\dim \St=q.$ Given a character $\eta$ of $\RU_2(k)$, form the tensor product $\eta\otimes \St$, which is still an irreducible representation of $\RU_2(k)$ of dimension $q$. We then get another family of irreducible representations of $\RU_2(k)$ given by $\wpair{\eta\otimes \St,\eta\in \widehat E^1}$ and there are total $q+1$ of them. Note that if $\chi$ is a character of $E^\times$ with $\chi|_{k^\times}=1$, then $I(\chi)=\eta\otimes I(1)=\eta \bigoplus (\eta\otimes \St)$, where $\eta$ is the character of $E^1$ determined by $\eta(a/\bar a)=\chi(a)$ for $a\in E^\times$. For simplicity, we write $\eta\otimes \St$ as $\St_\eta$. The following is the character table of the representations $\eta, I(\chi), \St_\eta$ for $\eta\in \widehat E^1, \chi\in \widehat E^\times, \chi\ne \bar \chi^{-1}$:
\begin{align*}
\begin{array}{|c|c|c|c|c|}
\hline
 & \ch_{\eta} & \ch_{I(\chi)} & \ch_{\St_\eta}   \\
\hline
\begin{pmatrix}x &\\ &x \end{pmatrix},x\in E^1  & \eta(x^2)&(q+1)\chi(x)& q\eta(x^2) \\
\hline
\begin{pmatrix}x&x\\ &x \end{pmatrix} ,x\in E^1 & \eta(x^2) &\chi(x)& 0 \\
\hline
 \begin{pmatrix}x &\\ &\bar x^{-1}\end{pmatrix},x\in E^\times-E^1& \eta(x\bar x^{-1}) &\chi(x)+\chi(\bar x^{-1})&\eta(x\bar x^{-1})\\
 \hline
\begin{pmatrix}x&y\\ \kappa y &x \end{pmatrix}, y\ne 0&\eta(x^2-\kappa y^2)&0&-\eta(x^2-\kappa y^2)\\
 \hline
\end{array}
\end{align*}

\subsection{Cuspidal representation}
Recall that $\psi$ is a fixed non-trivial additive character of $k$, and we identify $\psi$ as a character of $N_{\RU_2}$ by the isomorphism $N_{\RU_2}\cong k$.

Let $\mu$ be a non-trivial character of $E^1$ and let
$$\CW(\mu)=\wpair{f:E \ra \BC: f(yx)=\mu^{-1}(y)f(x),\forall x\in E, y\in E^1}.$$
 Recall that we have a representation $\omega_{\psi,\mu}$ of $\SL_2(k)$ on $\CW(\mu)$. Let $\eta$ be a character of $E^1$, one can extend the representation $\omega_{\psi,\mu}$ to a representation $\omega_{\psi,\mu,\eta}$ of $\RU_2(k)$ such that 
$$\omega_{\psi,\mu,\eta}\left(\bpm a&\\ &\bar a^{-1} \epm \right)\phi(x)=\eta(a\bar a^{-1})\phi(xa),\phi\in \CW(\mu).$$
Since any $g\in \RU_2(k)$ can be written as $g=\diag(a,\bar a^{-1})g_1$ for some $a\in E^\times,g_1\in \SL_2(k)$, the above relation uniquely determines an extension of $\omega_{\psi,\mu}$. It is not hard to check that the extension $\omega_{\psi,\mu,\eta}$ is indeed a representation. Let $1$ be the trivial representation of $E^1$, then $\omega_{\psi,\mu,\eta}=\eta\otimes \omega_{\psi,\mu,1}$. We now compute the character of $\omega_{\psi,\mu,\eta}$. 

For each $a\in k^\times$, we fix an element $x_a\in E^\times$ with $\Nm(x_a)=a$. Let $f_a\in \CW(\mu)$ be the function such that $f_a(x_b)=\delta_{a,b}$. Then $\wpair{f_a,a\in k^\times}$ becomes a basis of $W(\mu)$. Using this basis, we can compute the following character table:

\begin{align*}
\begin{array}{|c|c|c|}
\hline
 & \ch_{\omega_{\psi,\mu,\eta}}   \\
\hline
\begin{pmatrix}x &\\ &x \end{pmatrix},x\in E^1  & (q-1)\eta(x^2)\mu^{-1}(x) \\
\hline
\begin{pmatrix}x&x\\ &x \end{pmatrix} ,x\in E^1 & -\eta(x^2)\mu^{-1}(x)  \\
\hline
 \begin{pmatrix}x &\\ &\bar x^{-1}\end{pmatrix},x\in E^\times-E^1& 0 \\
 \hline
\begin{pmatrix}x&y\\ \kappa y &x \end{pmatrix}, y\ne 0&-\eta(x^2-\kappa y^2)(\mu^{-1}(x+y\sqrt{\kappa})+\mu^{-1}(x-y\sqrt{\kappa}))\\
 \hline
\end{array}
\end{align*}
From this table, we can check that the representation $\omega_{\psi,\mu,\eta}$ is irreducible if $\mu$ is non-trivial (which is always assumed). On the other hand, we have $\omega_{\psi,\mu,\eta}\cong\omega_{\psi,\mu_1,\eta_1}$ if and only if $(\mu_1,\eta_1)=(\mu,\eta)$ or $(\mu^{-1},\eta\mu^{-1})$. Thus there are totally $(q+1)q/2$ representations of the form $\omega_{\psi,\mu,\eta}$. The following family
$$ \eta, \St_\eta, I(\chi), \omega_{\psi,\mu,\eta}, $$
with $\mu,\eta\in \widehat E^1,\mu\ne 1, \chi\in \widehat E^\times,\chi\ne \bar \chi^{-1}$,  is a complete list of irreducible representations of $\RU_2(k)$.


\subsection{The Weil representation}\label{sec7.3}
Let $W=E\oplus E$, endowed with the skew-Hermitian structure 
$$\pair{u,v}=uJ_2 {}^t\bar v,$$
where $u,v$ are viewed as row vectors. We consider the Heisenburg group $\sH=W\oplus k$ with addition
$$[u_1,t_1]+[u_2,t_2]=[u_1+u_2,t_1+t_2+\frac{1}{2}\Tr_{E/k}(\pair{u_1,u_2})],u_1,u_2\in W,t_1,t_2\in k.$$
The group $\RU_2(k)$ act on $\sH$ by $g.[u,t]=[gu,t],g\in \RU_2(k),u\in W,t\in F$. Thus we can form the semi-direct product $\RU_2(k)\ltimes \sH$. 

 There is a Weil representation $\omega_{\psi}$ of $\RU_2(k)\ltimes \sH$ on $\CS(E)$ determined by the formulas 
\begin{align*}
\omega_\psi\left([x,0,z] \right)f(\xi)&=\psi(z)f(\xi+x),x\in E,z\in k,\\
\omega_\psi([0,y,0])f(\xi)&=\psi(\Tr(\bar y \xi))f(\xi), y\in E,\\
\omega_{\psi}\left( \begin{pmatrix}a&\\ &\bar a^{-1} \end{pmatrix}\right)f(\xi)&=f(a\xi), a\in E^\times,\\
\omega_{\psi}\left( \begin{pmatrix}1&b\\ &1 \end{pmatrix}\right)f(\xi)&=\psi(\Nm(\xi)b)f(\xi), b\in k,\\
\omega_{\psi}(w)f(\xi)&=-q^{-1}\sum_{y\in E}\psi(\Tr(\bar y \xi))f(y).
\end{align*}
Here $w=\bpm &1\\-1& \epm$ as usual. The Weil representation for general unitary group over finite fields is constructed in \cite{Ge}. The above formulas could be found in \cite[Section 4.1]{Bu}.

We next consider the restriction $\omega_\psi|_{\RU_2(k)}$. Let $\CW(1)$ be the subspace of $\CS(E)$ which consists of functions $f\in \CS(E)$ such that $f(ux)=f(x),\forall u\in E^1,x\in E.$ Then it is not hard to check that as representations of $\RU_2(k)$,  $\omega_{\psi}|_{\CW(1)}\cong \St$. Thus we have the following decomposition as in \cite[Corollary 4.5]{Ge} $$\omega_{\psi}|_{\RU_2(k)}=\St\bigoplus \left(\bigoplus_{\mu \in \widehat E^1,\mu\ne 1}\omega_{\psi,\mu,1} \right).$$
From this decomposition, the character table of $\omega_{\psi}$ is as following
\begin{align*}
\begin{array}{|c|c|c|}
\hline
 & \ch_{\omega_{\psi}}   \\
\hline
\bpm 1&\\ &1 \epm & q^2\\
\hline
\begin{pmatrix}x &\\ &x \end{pmatrix},x\in E^1,x\ne 1  & 1 \\
\hline
\bpm 1&1\\ &1 \epm & -q\\
\hline
\begin{pmatrix}x&x\\ &x \end{pmatrix} ,x\in E^1,x\ne 1 & 1 \\
\hline
 \begin{pmatrix}x &\\ &\bar x^{-1}\end{pmatrix},x\in E^\times-E^1& 1\\
 \hline
\begin{pmatrix}x&y\\ \kappa y &x \end{pmatrix}, x\pm y\sqrt{\kappa}=1,y\ne 0&-q\\
 \hline
\bpm x& y\\ \kappa y &x \epm, x\pm y\sqrt \kappa\ne 1, y\ne 0 & 1\\
\hline
\end{array}
\end{align*}
One can compare the above table with \cite[Theorem 4.5, Corollary 4.8.2 and Theorem 4.9.2]{Ge}.

\subsection{On the tensor product of an irreducible representation with the Weil representation}\label{sec7.4} In this section, we consider the decomposition of $\pi\otimes \omega_\psi$ for an irreducible representation $\pi$ of $\RU_2(k)$. 

Let $A$ be a set of representatives of characters of $E^\times$ such that $\chi|_{k^\times}\ne 1$ modulo the relation $\chi=\bar \chi^{-1}$. Then the cardinality of $A$ is $(q+1)(q-2)/2$. Let $B$ be a set of representatives of pairs $(\mu,\eta)\in \widehat E^1\times \widehat E^1$ with $\mu \ne 1$ modulo the relation  $(\mu, \eta)=(\mu^{-1},\eta \mu^{-1})$. Then the cardinality of $B$ is $(q+1)q/2$. The following proposition is $\RU_2(k)$ analogue of Proposition \ref{prop2.1} and we omit its proof.

\begin{prop}\label{prop7.2}
Let $\chi_1\in \widehat E^\times$ with $\chi_1\ne \bar \chi_1^{-1}$, and $\eta_1,\mu_1\in \widehat E^1$ with $\mu_1\ne 1$. We have
\begin{align*}
\eta_1\otimes\omega_{\psi}&=\St_{\eta_1}\bigoplus \left( \bigoplus_{\mu\in \widehat E^1,\mu \ne 1} \omega_{\psi,\mu,\eta_1} \right),\\
I(\chi_1)\otimes\omega_{\psi}&=2I(\chi_1)\bigoplus \left(\bigoplus_{\chi\in A,\chi\ne \chi_1,\bar \chi_1^{-1}} I(\chi) \right)\bigoplus \left(\bigoplus_{\eta\in \widehat E^1}\St_\eta \right) \bigoplus\left( \bigoplus_{(\mu,\eta)\in B} \omega_{\psi,\mu,\eta} \right),\\
\St\otimes\omega_\psi&=1\bigoplus \left(\bigoplus_{ \chi\in A } I(\chi) \right)\bigoplus\left(\bigoplus_{\eta\in \widehat E^1}\St_\eta \right)\bigoplus\left(\bigoplus_{(\mu,\eta)\in B,\eta\ne 1,\eta \ne \mu}\omega_{\psi,\mu,\eta}\right),\\
\omega_{\psi,\mu_1,1}\otimes\omega_\psi&=1\bigoplus \mu_1^{-1}\bigoplus\left(\bigoplus_{\chi\in A}I(\chi) \right)\bigoplus\left(\bigoplus_{\eta\ne 1, \eta \ne \mu}\St_\eta \right)\\
&\qquad \bigoplus\left(\bigoplus_{(\mu,\eta)\in B, \eta\ne 1, \eta\ne \mu, \eta\ne \mu_1,\eta\ne \mu_1\mu}\omega_{\psi,\mu,\eta} \right)
\end{align*}
\end{prop}
Note that $\St_\eta\otimes\omega_\psi=\eta\otimes \St\otimes\omega_\psi, $ and $\omega_{\psi,\mu,\eta}\otimes\omega_\psi=\eta\otimes\omega_{\psi,\mu,1}\otimes \omega_\psi$, the decomposition of $ \St_\eta\otimes\omega_\psi$ (resp. $\omega_{\psi,\mu,\eta}\otimes\omega_\psi $) can be given using the decomposition of $ \St\otimes\omega_\psi$ (resp. $\omega_{\psi,\mu,1}\otimes\omega_\psi $).

\subsection{A multiplicity one result for \texorpdfstring{$\RU_4(k)$}{Lg}}
In the group $\RU_4(k)$, we consider the subgroup 
$$J=\wpair{\bpm 1& *&*\\ &g&*\\ &&1 \epm, g\in \RU_2}.$$
Then there is an isomorphism $ \RU_2\ltimes \sH\ra J$ defined by
$$(g,[v,z])\ra \begin{pmatrix}1&\! v&z-\frac{1}{2}\bar x y+\frac{1}{2}x\bar y\\ &g&v^*\\ &&1 \end{pmatrix}, g\in \RU_2, v=(x,y)\in E^2,z\in k$$
where  $v^*=\begin{pmatrix}\bar y\\ -\bar x \end{pmatrix}$. Under this isomorphism, we view $\omega_\psi$ as a representation of $J$. Given an irreducible representation $\pi$ of $\RU_2(k)$, we consider the tensor product representation $\pi\otimes \omega_\psi$ of $J$.  Similar to the $\Sp_4(k)$ case, we have the following
\begin{thm}\label{thm: multiplicity one for U4}
Let $\pi$ be an irreducible representation of $\RU_2$ which is not of the form $I(\chi)$. Then the induced representation $\Ind_J^{\RU_4}(\pi\otimes \omega_\psi)$ is multiplicity free.
\end{thm}
Since the proof is similar to that of Theorem \ref{thm4.1}, we omit the details.


\section{The group \texorpdfstring{$\bG_2$}{Lg}}

In this section, we introduce the multiplicity one problem of certain Fourier-Jacobi models for the split exceptional group $\bG_2(k)$, which is quite similar to what we considered in Section \ref{sec3} for $\Sp_4(k)$. In this section, the notations $J, P, M,$ \textit{etc.}, will be used as subgroups of $\bG_2(k)$ rather than subgroups of $\Sp_4(k)$ as in Sections \ref{sec3} and \ref{sec4}. For simplicity, denote $G=\bG_2(k)$. 

\subsection{Roots and commutator relations}
 The group $\bG_2$ has two simple roots, the short root $\alpha$ and the long root $\beta$. The set of the positive roots is $\Sigma^+=\wpair{\alpha, \beta, \alpha+\beta, 2\alpha+\beta, 3\alpha+\beta, 3\alpha+2\beta}$.  Let $(~,~)$ be the inner product in the root system and $\pair{~,~}$ be the pair defined by $\pair{\gamma_1,\gamma_2}=\frac{2(\gamma_1,\gamma_2)}{(\gamma_2,\gamma_2)}$. For $\bG_2$, we have the relations:
$$ \pair{\alpha,\beta}=-1, \pair{\beta, \alpha}=-3.$$
For a root $\gamma$, let $s_\gamma$ be the reflection defined by $\gamma$, i.e., $s_\gamma(\gamma')=\gamma'-\pair{\gamma',\gamma}\gamma$. We have the relation
$$s_\alpha(\beta)=3\alpha+\beta, s_\beta(\alpha)=\alpha+\beta.$$
Let $W(\bG_2)$ be the Weyl group of $\bG_2$, which is generated by $s_\alpha,s_\beta$ and has size 12.

We use the following standard notations from Chevalley group theory (see \cite{St}). For a root $\gamma$, let $U_\gamma \subset G$ be the root space of $\gamma$, and  let $\bx_\gamma: k\ra U_\gamma$ be a fixed isomorphism which satisfies various Chevalley relations (see \cite[Chapter 3]{St}). Following \cite{St}, for $t\in k^\times$, denote $w_\gamma(t)= \bx_{\gamma}(t)\bx_{-\gamma}(-t^{-1})\bx_{\gamma}(t)$ and $w_\gamma=w_\gamma(1)$. Note that $w_\gamma$ is a representative of $s_\gamma$. Let $h_\gamma(t)=w_\gamma(t) w_\gamma^{-1}$.  For simplicity, we denote $$w_1=w_\alpha w_\beta w_\alpha^{-1}, w_2=w_\alpha w_\beta w_\alpha w_\beta^{-1}w_\alpha^{-1}.$$

Let $T$ be the subgroup of $G$ which consists of  elements of the form $h_\alpha(t_1)h_\beta(t_2), t_1,t_2\in k^\times$ and $U$ be the subgroup of $G$ generated by $U_\gamma$ for all $\gamma\in \Sigma^+$. Let $B=TU$, which is a Borel subgroup of $G$. 
 
For $t_1,t_2\in F^\times,$ denote $h(t_1,t_2)=h_\alpha(t_1t_2)h_\beta(t_1^2t_2)$. We can check the following relations
\begin{equation}
\begin{split}\label{eq5.1}
h^{-1}(t_1,t_2)\bx_\alpha(r)h(t_1,t_2)&=\bx_\alpha (t_2^{-1} r), \\
 h^{-1}(t_1,t_2)\bx_\beta(r)h(t_1,t_2)&=\bx_\beta(t_1^{-1}t_2 r).
\end{split}
\end{equation}
The notation $h(a,b)$ agrees with that of \cite{Gi}, and our $h(a,b)$ is $h(a,b,a^{-1}b^{-1})$ in the notation of \cite{CR}. One can also check that 
\begin{equation} \label{eq5.2}w_\alpha h(t_1,t_2) w_\alpha^{-1}= h(t_1 t_2, t_2^{-1}), \quad w_\beta h(t_1,t_2) w_\beta^{-1}= h(t_2,t_1).\end{equation}
For $g_1,g_2\in G$, denote $[g_1,g_2]=g_1^{-1}g_2^{-1}g_1g_2$. We have the following commutator relations (see \cite[p.443]{Re}):
\begin{equation}
\begin{split}\label{eq5.3}
[\bx_\alpha(x), \bx_\beta(y)]&=  \bx_{\alpha+\beta}(-xy)\bx_{2\alpha+\beta}(-x^2 y) \bx_{3\alpha+\beta}( x^3 y)\bx_{3\alpha+2\beta}(-2x^3 y^2),\\
[\bx_\alpha(x), \bx_{\alpha+\beta}(y)]&= \bx_{2\alpha+\beta}( -2xy)\bx_{3\alpha+\beta}(3x^2 y) \bx_{3\alpha+2\beta}(3xy^2) , \\
[\bx_\alpha(x), \bx_{2\alpha+\beta}(y)]&=\bx_{3\alpha+\beta}( 3xy),  \\
[\bx_\beta(x), \bx_{3\alpha+\beta}(y)]&=\bx_{3\alpha+2\beta}(xy),  \\
[\bx_{\alpha+\beta}(x), \bx_{2\alpha+\beta}(y)]&=\bx_{3\alpha+2\beta}(3xy).
\end{split}
\end{equation}
For all the other pairs of positive roots $\gamma_1, \gamma_2$, we have $[\bx_{\gamma_1}(x), \bx_{\gamma_2}(y)]=1$. 

We also need the Chevalley relation $w_{\gamma_1} \bx_{\gamma_2}(r) w_{\gamma_1}^{-1}=\bx_{w_{\gamma_1}(\gamma_2)}(c(\gamma_1,\gamma_2) r)$ (see \cite[Lemma 20, (b)]{St}), where $c(\gamma_1,\gamma_2)\in \wpair{\pm1}$ and  $c(\gamma_1,\gamma_2)=c(\gamma_1,-\gamma_2)$. The numbers $c(\gamma_1,\gamma_2)$ are given in the following:
\begin{align}\label{eq5.4} c(\alpha, \alpha)=c(\alpha, 2\alpha+\beta)=c(\alpha,3\alpha+\beta)=-1, &\quad  c(\alpha, \beta)=c(\alpha, \alpha+\beta)=c(3\alpha+2\beta)=1,\\
c(\beta,\beta)=c(\beta,\alpha+\beta)=c(\beta,3\alpha+2\beta)=-1, &\quad c(\beta,\alpha)=c(\beta,2\alpha+\beta)=c(\beta,3\alpha+\beta)=1.\nonumber
\end{align}

\subsection{Subgroups}\label{sec5.2}
 The group $G$ has two proper parabolic subgroups. Let $P=M\ltimes V$ be the parabolic subgroup of $G$ with Levi $M$ and unipotent $V$, such that $U_\beta\subset M\cong \GL_2(k)$. The isomorphism $M\cong \GL_2$ is determined by
\begin{align*}
\bx_\beta(r)\mapsto \begin{pmatrix}1& r\\ &1 \end{pmatrix},\quad
h(a,b)\mapsto \begin{pmatrix}a&\\ &b \end{pmatrix}.
\end{align*}
The unipotent subgroup $V$ of $P$ consists of root spaces of $\alpha, \alpha+\beta, 2\alpha+\beta, 3\alpha+\beta,3\alpha+2\beta$, and a typical element of $V$ is of the form 
$$ \bx_{\alpha}(r_1)\bx_{\alpha+\beta}(r_2)\bx_{2\alpha+\beta}(r_3)\bx_{3\alpha+\beta}(r_4)\bx_{3\alpha+2\beta}(r_5), r_i\in k.$$
To ease the notation, we write the above element as $(r_1,r_2,r_3,r_4,r_5)$. Denote by $J$ the following subgroup of $P$  
$$J=\SL_2(k)\ltimes V.$$
We always view $\SL_2(k)$ as a subgroup of $G$ via the embedding $\SL_2(k)\subset \GL_2(k)\cong M\incl G$.
Let $ V_1$ (resp. $Z$) be the subgroup of $V$ which consists of root spaces of $3\alpha+\beta$ and $3\alpha+2\beta$ (resp. $2\alpha+\beta, 3\alpha+\beta$ and $3\alpha+2\beta$). Note that $P$ and hence $J$ normalize $V_1$ and $Z$. 

\subsection{An anti-involution on \texorpdfstring{$G$}{Lg}}\label{sec5.3}
For $g\in G$, we define ${}^\iota g=h(1,-1)gh(1,-1)$ and ${}^\tau g={}^\iota g^{-1}$. Then we can check that 
$${}^\iota(r_1,r_2,r_3,r_4,r_5)=(-r_1,r_2,-r_3,r_4,-r_5),$$
and $${}^\iota \begin{pmatrix}a&b\\ c&d \end{pmatrix}=\begin{pmatrix}a&-b\\ -c&d \end{pmatrix}.$$
In particular, both the involution ${}^\iota$ and the anti-involution ${}^\tau$ preserve $J$.

\subsection{Weil representations and the problem}
Let $W=k^2$, endowed with the symplectic structure $\pair{~,~}$ defined by
\begin{equation}\label{eq5.5} \pair{(x_1,y_1),(x_2,y_2)}=-2x_1y_2+2x_2y_1. \end{equation}
Note that the symplectic form on $W$ here is different from the one defined in Section \ref{sec2.2}. The reason for choosing this non-standard symplectic structure on $W$ will be explained below. 

Let $\sH$ be the Heisenberg group associated with the symplectic space $W$. Explicitly, $\sH=W\oplus k$ with addition
$$[x_1,y_1,z_1]+[x_2,y_2,z_2]=[x_1+x_2,y_1+y_2,z_1+z_2-x_1y_2+x_2y_1].$$
Let $\SL_2(k)$ act on $\sH$ such that it acts on $W$ from the right and acts on the third component $k$ in $\sH$ trivially. Then we can form the semi-direct product $\SL_2(k)\ltimes \sH.$ The product map in $\SL_2(k)\ltimes \sH$ is given by 
$$(g_1,v_1)(g_1,v_2)=(g_1g_2, v_1.g_2+ v_2).$$
Let $\psi$ be a fixed non-trivial additive character of $k$, and let $\omega_\psi$ be the Weil representation of $\SL_2(k)\ltimes \sH$ on $\CS(k)$, where $\CS(k)$ is the space of $\BC$-valued functions on $k$. The formulas (\ref{eq2.1}) should be adapted to our new symplectic structure on $W$. 

Define a map
$\pr: V\ra \sH$
$$\pr((r_1,r_2,r_3,r_4,r_5))=[r_1,r_2,r_3-r_1 r_2] . $$
From the commutator relations in (\ref{eq5.3}), we can check that $\pr$ is a group homomorphism and defines an exact sequence
$$0\ra  V_1 \ra V\ra \sH\ra 0.$$
Here, to ensure that $\pr$ is a group homomorphism, we need to choose the symplectic form on $W$ in the non-standard way (\ref{eq5.5}). It seems that there is a typo in the formula of the projection map $\pr$ in \cite[p.316]{Gi}.

For $g=\begin{pmatrix}a & b\\ c& d \end{pmatrix}\in \SL_2(F) \subset M$, we have
$$g^{-1}(r_1,r_2,r_3,0,0) g=(r_1', r_2', r_3', r_4',r_5'),$$
where $r_1'=a r_1- cr_2, r_2'=-b r_1 + dr_2, r_3'-r_1'r_2'=r_3-r_1r_2$. This implies that the map $\overline \pr: J=\SL_2(F)\ltimes V\ra \SL_2(F)\ltimes \sH$, defined by, $$(g,v)\mapsto (g^*, \pr(v)), g\in \SL_2, v\in V,$$
is a group homomorphism,
where $g^*=\begin{pmatrix} a& -b \\ -c & d \end{pmatrix}=d_1 g d_1^{-1}$ and $d_1=\diag(-1,1)\in \GL_2(k)$. Now, we can view the Weil representation $\omega_\psi$ as a representation of $J$ by composing with the map $\ov{\pr}$. Given an irreducible representation $\pi$ of $\SL_2(k)$, view it as a representation of $J$ via the quotient map $J\ra \SL_2(k)$.

Similarly as the $\Sp_4(k)$ case, we consider the problem: for what irreducible representation $\pi$ of $\SL_2(k)$, the induced representation $\Ind_J^{\bG_2(k)}(\pi\otimes\omega_\psi)$ is multiplicity free? 
The answer is similar to the $\Sp_4(k)$ case and the proof will be given in next section.

\begin{rmk}\label{rmk5.1}
\rm{
If $k$ is a local field, let $\wt{\SL}_2(k)$ be the metaplectic double cover of $\SL_2(k)$. Then there is a Weil representation $\omega_\psi$ of $\wt{\SL}_2(k)\ltimes \sH$. By composing with the projection map, we could view $\omega_\psi$ as a representation of $\wt{\SL}_2(k)\ltimes V$. Given a genuine irreducible representation $\pi$ of $\wt{\SL}_2(k)$, the tensor product $\pi\otimes \omega_\psi$ can be viewed as a representation of $J=\SL_2(k)\ltimes V$.
Due to the similarity between the $\RG_2$ case and the $\Sp_4$ case, we propose the following
\begin{conj}\label{conj5.2}
For any self-dual irreducible representation $\sigma$ of $\bG_2(k)$ and any irreducible genuine representation $\pi$ of $\wt\SL_2(k)$, we have 
$$\dim\Hom_J(\sigma,\pi\otimes \omega_\psi)\le 1.$$
\end{conj}
Here the ``self-dual" condition might be removable and we add it due to certain technical difficulties in the application of the Gelfand-Kazhdan method. A nonzero element in $\Hom_J(\sigma,\pi\otimes \omega_\psi) $ defines an embedding $\sigma\incl \Ind_J^{\bG_2(k)}(\pi\otimes\omega_\psi)$. Such a realization of $\sigma$ will be called a {\it Fourier-Jacobi} model of $\sigma$ with respect to the datum $(\pi\otimes\omega_\psi,J)$. 
Given a character $\chi$ of $\GL_1(k)$, consider the genuine indueced representation 
$\tilde I(s,\chi)$ of $\wt\SL_2(k)$, $s\in \BC$, given an irreducible generic representation $\sigma$ of $\bG_2(k)$,
Ginzburg (\cite{Gi}) has contructed a local zeta integral which defines an elements in 
$\Hom_J(\sigma, \tilde I(s,\chi)\otimes\omega_\psi )$. 
Thus the above conjecture would imply the local functional equation for Ginzburg's local zeta integral in \cite{Gi}. We can prove that $\dim\Hom_J(\sigma, \tilde I(s,\chi)\otimes\omega_\psi )\le 1 $ for 
{\it non-supercuspidal} representations $\sigma$, which could be viewed as a special case and an evidence of the above conjecture.  \qed}
\end{rmk}

As preparations of our multiplicity results in next section, we record some useful facts on the Weil representations. From the above description of the projection map and modified versions\footnote{See \cite{Ku}, for example, for the dependence of these formulas on the symplectic form.} of Eq.(\ref{eq2.1}), we have the following formulas
\begin{align}
\begin{split}\label{eq5.7}
\omega_\psi((r_1,0,r_3,r_4,r_5))\phi(\xi)&=\psi(r_3)\phi(\xi+r_1),\\
\omega_\psi((0,r_2,0,0,0))\phi(\xi)&=\psi(-2\xi r_2)\phi(\xi),\\
\omega_\psi(h(a,a^{-1}))\phi(\xi)&=\epsilon(a)\phi(a\xi),\\
\omega_\psi (\bx_\beta(b))\phi(\xi)&=\psi(b\xi^2)\phi(\xi),\\
\omega_\psi\left(\begin{pmatrix}&b\\ -b^{-1}& \end{pmatrix} \right)\phi(\xi)&=\frac{1}{\gamma(b,\psi)}\sum_{x\in k} \phi(x)\psi(-2xb\xi).
\end{split}
\end{align}
The space $\CS(k)$ has a basis $\wpair{\delta_s,s\in k}$, where $\delta_s(t)=\delta_{s,t}$.
From formulas (\ref{eq5.7}), we have 
\begin{align}
\begin{split}\label{eq5.8}
\omega_\psi(\bx_\beta(b))\delta_s&=\psi(bs^2)\delta_s,\\
 \omega_\psi((0,y,0,0,0))\delta_s&=\psi(-2sy)\delta_s,\\
 \omega_\psi((r_1,0,r_3,r_4,r_5))\delta_s&=\psi(r_3)\delta_{s-r_1}.
\end{split}
\end{align}

\begin{lem}\label{lem5.3}
For $\phi,\phi'\in \CS(k)$, we define a pair
$$\pair{\phi,\phi'}=\sum_{\xi\in k}\phi(\xi)\phi'(-\xi).$$
Then we have 
$$\pair{\omega_\psi(j)\phi,\omega_\psi({}^\iota \! j)\phi'}=\pair{\phi,\phi'},\forall j\in J, \phi,\phi'\in \CS(k).$$
For $A\in \End_{\mathbb{C}}(\CS(k))$, we define ${}^t \! A$ by 
$$\pair{{}^t \! A\phi,\phi'}=\pair{\phi, A\phi'}.$$
The operator ${}^t$ is an anti-involution on $\End_{\mathbb{C}}(\CS(k))$ and satisfies ${}^t\omega_\psi(j)=\omega_\psi(^\tau \! j)$.
\end{lem}
\begin{proof}
The proof is similar to that of Lemma \ref{lem3.5}, and thus is omitted.
\end{proof}

Note that the above pair on $\CS(k)$ satisfies the property 
\begin{equation}\label{eq5.9}\pair{\delta_a,\delta_{b}}=\delta_{a,-b},a,b\in k. \end{equation}

\subsection{Transposes on \texorpdfstring{$\End(\pi\otimes \omega_\psi)$}{Lg} for an irreducible representation \texorpdfstring{$\pi$}{Lg} of \texorpdfstring{$\SL_2(k)$}{Lg}}\label{sec5.5} Let $\pi=1,\St,\omega_{\psi}^{\pm},\omega_{\psi_\kappa}^{\pm},$ or $\omega_{\psi,\mu}$, where $\mu$ is a character of $E^1$ such that $\mu^2\ne 1$. Then $\pi$ is an irreducible representation of $\SL_2(k)$. We have defined a pair $\pi\times {}^\iota\!\pi\ra \BC$ in Section \ref{sec3}. Considering the representation $\sigma_\pi:=\pi\otimes \omega_\psi$ of $J$, we can then define a bilinear pair $\sigma_\pi\times {}^\iota\!\sigma_\pi\ra \BC$ by 
$$\pair{v_1\otimes \phi_1,v_2\otimes \phi_2}=\pair{v_1,v_2}\pair{\phi_1,\phi_2},v_1,v_2\in \pi,\phi_1,\phi_2\in \omega_\psi,$$
where $\pair{\phi_1,\phi_2}$ is defined in Lemma \ref{lem5.3}. This pair satisfies the property 
$$\pair{\sigma_\pi(j)v,\sigma_\pi({}^\iota\! j)v'}=\pair{v,v'},v,v'\in \sigma_\pi,j\in J.$$ 
As in the $\Sp_4$ case, we define a transpose ${}^t: \End_{\mathbb{C}}(\sigma_\pi)\ra \End_{\mathbb{C}}(\sigma_\pi)$ by 
$$\pair{{}^t\!Av,v'}=\pair{v,Av'}, A\in \End_{\mathbb{C}}(\sigma_\pi),v,v'\in \sigma_\pi.$$
Then we have 
$${}^t\! \sigma_\pi(j)=\sigma_\pi({}^t\! j),\forall j\in J.$$
By Proposition \ref{prop2.1}, the representation $\sigma_\pi|_{\SL_2(k)}$ is multiplicity free. Let $\sigma_\pi|_{\SL_2(k)}=\oplus_i V_i$ be the decomposition of irreducible representations of $\SL_2(k)$.  As in Lemma \ref{lem3.9}, the idempotent $\id_{V_i}\in \End_{\mathbb{C}}(\sigma_\pi)$ is invariant under the transpose ${}^t$ defined above.

\section{Certain Multiplicity one theorems for \texorpdfstring{$\bG_2$}{Lg}}\label{sec6}
In this section, we continue to let $G=\bG_2(k)$ and let $J$ be the Fourier-Jacobi subgroup of $G$ defined in Section \ref{sec5.2}. 
For an irreducible representation $\pi$ of $\SL_2(k)$, as in Section \ref{sec5.5}, 
we write $\sigma_\pi=\pi\otimes \omega_\psi.$
\begin{thm}\label{thm6.1}
The representation $\Ind_J^{G}(\sigma_\pi)$ is multiplicity free, if $\pi=1, \St, \omega_{\psi,\mu}, \omega_{\psi}^{\pm},\omega_{\psi_\kappa}^{\pm}$ for any character $\mu$ of $E^1$ with $\mu^2\ne 1$.
\end{thm}

\begin{rmk}\label{rmk6.2}
\rm{
(1) If $\pi=I(\chi)$ for a character $\chi$ of $k^\times$ with $\chi^2\ne 1$, we can also show that for $q$ large, the induced representation $\Ind_J^G(I(\chi)\otimes \omega_\psi)$ is not multiplicity free as in Remark \ref{rmk4.2}. In fact, let $\chi_1,\chi_2$ be two characters of $k^\times$, we consider the induced representation $I(\chi_1,\chi_2)=\Ind_{B_{\GL_2}}^{\GL_2(k)}(\chi_1\otimes \chi_2)$ of $\GL_2(k)$, where $B_{\GL_2}$ is the upper triangular subgroup of $\GL_2(k)$, and $\chi_1\otimes \chi_2$ is viewed as a character of $B_{\GL_2}$ by 
$$\chi_1\otimes \chi_2\left(\bpm a_1&b\\ &a_2\epm \right)=\chi_1(a_1)\chi_2(a_2).$$
 Recall that $P=M\ltimes V$ with $M\cong \GL_2(k)$. View $I(\chi_1,\chi_2)$ as a representation of $P$ by making $V$ act trivially on it. Then we consider the induced representation $\Ind_P^G(I(\chi_1\otimes \chi_2))$ and  $$\Hom_G(\Ind_J^G(I(\chi)\otimes \omega_\psi),\Ind_P^G(I(\chi_1,\chi_2))).$$
By Frobenius reciprocity, we have 
\begin{align*}
    \Hom_G(\Ind_J^G(I(\chi)\otimes \omega_\psi),\Ind_P^G(I(\chi_1,\chi_2)))&=\Hom_J(I(\chi)\otimes \omega_\psi,\Ind_P^G(I(\chi_1,\chi_2))|_J).
\end{align*}
By Mackey's Theory, see \cite[p.58]{Se}, we have 
$$\Ind_P^G(I(\chi_1,\chi_2))|_J=\bigoplus_{s\in J\backslash G/P}\Ind_{P_s}^J(I(\chi_1,\chi_s)^s), $$
where $P_s=sPs^{-1}\cap J$ and for a representation $\rho$ of $P$, the representation $\rho^s$ of $P_s$
 is defined by $\rho^s(h)=\rho(s^{-1}hs)$. Considering the element $w_2=w_\alpha w_\beta w_\alpha w_\beta^{-1}w_\alpha^{-1}\in J\backslash G/P$, we have $P_{w_2}\cong \SL_2(k)\incl M$ and $I(\chi_1,\chi_2)^{w_2}=I(\chi_1,\chi_2)|_{\SL_2(k)} $. Thus 
 \begin{align*}
    \Hom_G(\Ind_J^G(I(\chi)\otimes \omega_\psi),\Ind_P^G(I(\chi_1,\chi_2)))&\supset\Hom_J(I(\chi)\otimes \omega_\psi,\Ind_{P_{w_2}}^J(I(\chi_1,\chi_2)|_{\SL_2(k)}))\\
    &=\Hom_{\SL_2(k)}(I(\chi)\otimes \omega_\psi|_{\SL_2(k)}, I(\chi_1,\chi_2)|_{\SL_2(k)} ).
\end{align*}
 Note that $I(\chi_1,\chi_2)|_{\SL_2(k)} =I(\chi_1\chi_2^{-1})$. We take $\chi_1=\epsilon \chi \chi_2$. By Proposition \ref{prop2.1}, we have
 $$ \dim \Hom_G(\Ind_J^G(I(\chi)\otimes \omega_\psi),\Ind_P^G(I(\epsilon \chi\chi_2 ,\chi_2)))\ge \dim \Hom_{\SL_2(k)}(I(\chi)\otimes \omega_\psi|_{\SL_2(k)}, I(\epsilon \chi) )=2.$$
 
 From Mackey's irreducibility criterion, see \cite[p.59]{Se}, for $\chi_1,\chi_2$ in ``general positions", the induced representation $\Ind_P^G(I(\chi_1,\chi_2))$ is irreducible. Here $\chi_1,\chi_2$ are called in general position, if $(\chi_1\otimes \chi_2)\ne (\chi_1\otimes \chi_2)^w$ for all $w\in W(G)-\wpair{1}$, where $(\chi_1\otimes\chi_2)$ is viewed as a character of the maximal torus of $G$ via $(\chi_1\otimes\chi_2)(h(a,b))=\chi_1(a)\chi_2(b)$. One can check that, for $q$ large, it is easy to find $\chi_2$ such that $\epsilon \chi\chi_2,\chi_2$ are in general positions. Then $\Ind_P^G(I(\epsilon\chi \chi_2,\chi_2))$ is irreducible and $\dim \Hom_G(\Ind_J^G(I(\chi)\otimes \omega_\psi),\Ind_P^G(I(\epsilon\chi \chi_2,\chi_2)))\ge 2 $. Thus $\Ind_J^G(I(\chi)\otimes \omega_\psi)$ is not multiplicity free. 
 
  (2) For an irreducible representation $\Pi$ of $G$, it is in general false that $\dim\Hom_J(\Pi,I(\chi)\otimes \omega_\psi)\le 1$ by the above discussion. However, if we require further that $\Pi$ is \textit{cuspidal}, in \cite{LZ}, we are able to show that the multiplicity one result $\dim\Hom_J(\Pi,I(\chi)\otimes \omega_\psi)\le 1 $ still holds. 
  This multiplicity one result will be used in \cite{LZ} to define twisted gamma factors for irreducible generic cuspidal representations of $G$. 
  \qed}
 \end{rmk}

 Before proving Theorem \ref{thm6.1}, we give a set of representatives of the double coset decomposition $J\backslash G/J$. Recall that $w_1=w_\alpha w_\beta w_\alpha^{-1}, w_2=w_\alpha w_\beta w_\alpha w_\beta^{-1}w_\alpha^{-1}.$

\begin{lem}\label{lem6.3}
A set of representatives of $J\backslash G/J$ is given by
$$\wpair{h(a,1), h(a,a^{-2})w_\alpha, h(-1,a)w_1, h(a,1)w_2,a\in k^\times}.$$
\end{lem}
\begin{proof}
Considering the double coset decomposition $P\backslash G/P$, we have 
$$G=Pw_2P \cup Pw_1P \cup Pw_\alpha P\cup P.$$
From the fact $P=\cup_{a\in k^\times}h(a,1)J=\cup_{b\in k^\times}Jh(1,b)$, we can get the statement of the lemma. Note that $h(a,1/a)\in J$. Thus $h(a,b)w$, $h(ab,1)w$ and $h(1,ab)w$ are in the same double coset $Jh(a,b)wJ$ for any Weyl element $w$. We take $h(a,a^{-2})w_\alpha$ and $h(-1,a)w_1$ as representatives because we have ${}^\tau\! (h(a,a^{-2})w_\alpha)=h(a,a^{-2})w_\alpha$ and ${}^\tau \! (h(-1,a)w_1)=h(-1,a)w_1$.
\end{proof}

Before we start the proof of Theorem \ref{thm6.1}, we show one more lemma as follows. Let $(\rho,V)$ be an irreducible representation of $\SL_2(k)$ and $a\in k^\times$. Let $(\rho^a,V^a)$ be the representation of $\SL_2(k)$ defined by $V^a=V,\rho^a(g)=\rho(g^a)$, where $g^a=\diag(a,1)g\diag(a^{-1},1)$. For each $(\rho,V)$, we fix a non-trivial pair $\pair{~,~}$ on $V$ such that 
$$\pair{\rho(g)v,\rho({}^\iota\!g) v'}=\pair{v,v'},\forall v,v'\in V,g\in \SL_2(k).$$
Note that defining such a pair is equivalent to defining an isomorphism ${}^\iota\!\rho\cong \tilde \rho$, and thus such a pair is unique up to a scalar. 

\begin{lem}\label{lem6.4}
 If $\Hom_{\SL_2(k)}(\rho,\rho^a)\ne 0$, there exists a unique nonzero $\lambda^a\in \Hom_{\SL_2(k)}(\rho,\rho^a)$ such that 
$$\pair{\lambda^a v_1,\lambda^a v_2}=\pair{v_1,v_2},\forall v_1,v_2\in V,$$
where $\lambda^a v$ is viewed as an element in $V$ under the identification $V=V^a$ for $v\in V$.

Moreover, for the unique $\lambda^a$ defined above, let $d^a_V$ be the constant such that $\rho(\diag(a^{-1},a))\circ \lambda^a\circ \lambda^a=d_V^a\id_V$, then $d_V^a=1$. Here, the middle $\lambda^a$ is viewed as an element of $\Hom_{\SL_2(k)}(\rho^a,\rho^{a^2})$.
\end{lem}
\begin{proof}
We first fix any nonzero $\lambda_1^a\in \Hom_{\SL_2(k)}(\rho,\rho^a)$ and consider the pair $\pair{~,~}_a$ on $V$ defined by 
$$\pair{v_1,v_2}_a:=\pair{\lambda_1^av_1,\lambda_1^a v_2}.$$
For $g\in \SL_2(k)$, one can check that 
\begin{align*}
  \pair{\rho(g)v_1,\rho({}^\iota g)v_2}_a&=\pair{\lambda_1^a \rho(g)v_1,\lambda_1^a \rho({}^\iota\!g)v_2}\\
  &=\pair{\rho(g^a)(\lambda_1^a v_1),\rho(({}^\iota\! g)^a)(\lambda_1^av_2)}\\
  &=\pair{\rho(g^a)(\lambda_1^av_1),\rho({}^\iota(g^a))(\lambda_1^a v_2)}\\
  &=\pair{\lambda_1^av_1,\lambda_1^av_2}\\
  &=\pair{v_1,v_2}_a,
\end{align*}
where we used the relation ${}^\iota\!(g^a)=({}^\iota\!g)^a$. Thus by the  uniqueness of the pair, we have that there exists a constant $c_a\in \BC^\times$ such that $\pair{v_1,v_2}_a=c_a \pair{v_1,v_2}$. Then $\lambda^a:=\frac{1}{\sqrt{c_a}}\lambda_1^a\in \Hom_{\SL_2(k)}(\rho,\rho^a)$ satisfies the property
$$\pair{\lambda^av_1,\lambda^av_2}=\pair{v_1,v_2},\forall v_1,v_2\in V,$$
where $\sqrt{c_a}$ is a square root of $c_a$. The uniqueness of such $\lambda^a$ follows from the fact that $$\dim\Hom_{\SL_2(k)}(\rho,\rho^a)=1.$$ 

The ``moreover" part seems very delicate and we don't have a uniform proof at this moment. We will check it case by case.

We first consider that $(\rho,V)=\omega_{\psi,\mu}$ for a quadratic character $\mu$ of $E^1$ with $\mu^2\ne 1$. For $a\in F^\times$, we fix $x_a\in E^\times$ with $\Nm(x_a)=x_a^{q+1}=a$. Recall that an element $f\in V$ is a function on $E^\times$ such that $f(uy)=\mu^{-1}(u)f(y),\forall u\in E^1,y\in E^\times$. For $f\in V$, we consider $\lambda^af(\xi)=\sqrt{\mu^{-1}(x_a^{q-1})}f(x_a\xi).$ Then we can check that $\pair{\lambda^a f_1,\lambda^a f_2}=\pair{f_1,f_2}$, where $\pair{f_1,f_2}$ is the pair defined in $\S$\ref{sec3.7}. Moreover,
\begin{align*}
\lambda^a\circ \lambda^a f(\xi)
&=\mu^{-1}(x_a^{q-1})f(x_a^2\xi)\\
&=\mu^{-1}(x_a^{q-1})\mu^{-1}(x_a^2a^{-1})f(a\xi)\\
&=\mu^{-1}(x_a^{q+1}a^{-1})f(a\xi)\\
&=f(a\xi)\\
&=\rho(\diag(a,a^{-1}))f(\xi).
\end{align*}
Thus $d_V^a=1$ in this case.

Next, we consider the case when $(\rho,V)=I(\chi)$ for a character $\chi$ of $k^\times$. If $\chi$ is not quadratic, we haven't constructed a $\pair{~,~}$ on $I(\chi)$ in previous sections. We first define a pair below. Consider the intertwining operator $\Delta:I(\chi)\ra I(\chi^{-1})$ defined by 
$$\Delta*f(g)=q^{-1}\sum_{x\in k}f\left(w\begin{pmatrix}1&x\\&1 \end{pmatrix}g\right),\forall f\in I(\chi),$$
where $w=\bpm &1\\-1&\epm.$
One can check that $\Delta*f\in I(\chi^{-1})$. For $f_1,f_2\in I(\chi)$, we define 
$$\pair{f_1,f_2}=\sum_{g\in B_{\SL_2(k)}\backslash \SL_2(k)}\Delta*f_1(g)f_2({}^\iota\!g).$$
Then, one can check that 
$$\pair{\rho(g)f_1,\rho({}^\iota\!g)f_2}=\pair{f_1,f_2},\forall f_1,f_2\in I(\chi),g\in \SL_2(k).$$
For $a\in k^\times$, $f\in I(\chi)$, we define 
$$\lambda^af(g)=\sqrt{\chi(a)}f(g^{a^{-1}}).$$
Then $\lambda^a\in \Hom_{\SL_2(k)}(V,V^a)$, and
$$\Delta*\lambda^af(g)=\chi(a^{-1})\Delta*f(g^{a^{-1}}).$$
It follows that 
$$\pair{\lambda^af_1,\lambda^af_2}=\pair{f_1,f_2}.$$
On the other hand, we have 
\begin{align*}
    \lambda^a\circ \lambda^af(g)&=\chi(a)f(g^{a^{-2}})\\
    &=\chi(a)f(\diag(a^{-1},a)g\diag(a,a^{-1}))\\
    &=f(g\diag(a,a^{-1}))\\
    &=\rho(\diag(a,a^{-1}))f(g).
\end{align*}
Thus we get $d_V^a=1$. 

We omit the cases when $\rho=1,\St,$ and just remark that the proofs in these cases are similar.
\end{proof}

In the following, for $a\in k^\times$, we write $\eta(a)=h(a,a^{-2})w_\alpha$, and $\xi(a)=h(-1,a)w_1$, for simplicity.

\noindent\textbf{Proof of Theorem $\ref{thm6.1}$.}
We have defined an anti-involution ${}^\tau$ on $G$ by ${}^\tau\! g=h(1,-1)g^{-1}h(1,-1)$, see Section \ref{sec5.3}. In Section \ref{sec5.5}, we have constructed an anti-involution ${}^t$ on $\End(\sigma_\pi)$ such that ${}^t\! (\sigma_\pi(j))=\sigma_\pi({}^\tau \! j)$. We define an anti-involution ${}^\tau$ on $\CA(G,J,\sigma_\pi)$ by $$({}^\tau \! K)(g)={}^t\! ( K({}^\tau\! g)),\forall g\in G, K\in \CA(G,J,\sigma_\pi).$$
By Corollary \ref{cor3.4} and Lemma \ref{lem6.3}, it suffices to show that for all $K\in \CA(G,J,\omega_\psi)$ with ${}^\tau \! K=-K$, $K(g)=0$ for all $g= h(a,1), \eta(a), \xi(a),h(a,1)w_2,$ $\forall a\in k^\times$. We now fix a $K\in \CA(G,J,\omega_\psi)$ with ${}^\tau \! K=-K$.

Step (1), $K(h(a,1))=0$.
We have $$h(a,1)\bx_{2\alpha+\beta}(z)=\bx_{2\alpha+\beta}(az)h(a,1),$$
see Eq.(\ref{eq5.1}). Thus we get 
$$\psi(z)K(h(a,1))=\psi(az)K(h(a,1)).$$
If $a\ne 1$, we can take $z$ such that $\psi(z)\ne \psi(az)$. Then we get $K(h(a,1))=0$ if $a\ne 1$. If $a=1$, we have $h(1,1)g=gh(1,1)$ for all $g\in J$. Thus we get $K(h(1,1))= C\sigma_\pi(1)$ for some $C\in \BC$ by Schur's Lemma. We can get $C=0$ from the condition ${}^\tau \! K=-K.$

Step (2), $K(\eta(a))=0.$ From Eqs.(\ref{eq5.1})-(\ref{eq5.4}), we can check the relations
\begin{align*}\eta(a) (0,x,y,0,0,0)&=(0,-ay,x/a,0,3 xy) \eta(a), \forall x,y\in k,\\
\eta(a) \bx_\beta(b)&=(0,0,0,a^{-3}b,0) \eta(a), \forall b\in k,\\
 \bx_\beta(x) \eta(a)&=\eta(a) (0,0,0,-x/a^3,0), \forall x\in k.
\end{align*}
Thus, we have that
\begin{align}\label{eq6.1}
K(\eta(a))\sigma_\pi((0,x,y,0,0,0))&=\sigma_\pi((0,-ay,a^{-1}x,0,3 xy))K(\eta(a)), \forall x,y\in k,\\
\label{eq6.2}K(\eta(a))\sigma_\pi(\bx_\beta(b))&=K(\eta(a)), \forall b\in k,\\
\label{eq6.3}\sigma_\pi(\bx_\beta(x))K(\eta(a))&=K(\eta(a)), \forall x\in k. \end{align} 
We now consider different $\pi$ separately. 

Case (2.1), $\pi=1$. Put $y=0$ in Eq.(\ref{eq6.1}) and apply Eq.(\ref{eq5.8}), we can get 
$$\psi(-2sx)K(\eta(a))\delta_s=\psi(a^{-1}x)K(\eta(a))\delta_s.$$
We then get $K(\eta(a))\delta_s=0$ if $s\ne -1/(2a)$. Applying Eq.(\ref{eq6.2}) and Eq.(\ref{eq5.8}), we get $\psi(bs^2)K(\eta(a))\delta_s=K(\eta(a))\delta_s$. Thus if $s\ne 0$, we get $K(\eta(a))\delta_s=0$. Since $-1/(2a)\ne 0$, we have $K(\eta(a))\delta_s=0$ for all $s\in k$.

Case (2.2), $\pi=\omega_{\psi,\mu}$. Recall that $\omega_{\psi,\mu}$ has a basis $\wpair{f_b,b\in k^\times }$, where $f_a(x_b)=\delta_{a,b}$, see Section \ref{sec3.7}. We can check that $\sigma_\pi(\bx_\beta(b))f_r\otimes \delta_s=\psi(b(r+s^2))f_r\otimes \delta_s.$ Plugging $y=0$ into Eq.(\ref{eq6.1}), we can get $\psi(-2xs)K(\eta(a))f_r\otimes \delta_s=\psi(a^{-1}x)K(\eta(a))f_r\otimes \delta_s$. Thus we get $K(\eta(a))f_r\otimes \delta_s=0$ if $s\ne -1/(2a)$. By Eq.(\ref{eq6.2}), we can get $\psi(b(r+s^2))K(\eta(a))f_r\otimes \delta_s=K(\eta(a))f_r\otimes \delta_s$. Thus we get $K(h(a,1)w_\alpha)f_r\otimes \delta_s=0$ if $r\ne -s^2$. We assume that 
$$K(h(a,1)w_\alpha)f_{-1/(4a^2)}\otimes \delta_{-1/(2a)}=\sum_{b\in k^\times,t\in k}C(b,t)f_b\otimes \delta_t$$ for $C(b,t)\in \BC$. Applying Eq.(\ref{eq6.1}) when $x=0$, we can get
$$\sum_{b,t}\psi(y)C(b,t)f_b\otimes \delta_t=\sum_{b,t}\psi(2ayt)C(b,t)f_b\otimes \delta_t. $$
Thus we get $C(b,t)=0$ if $t\ne 1/(2a)$. Applying Eq.(\ref{eq6.3}), we can get $C(b,1/(2a))=0$ if $b\ne -1/(4a^2)$. We denote $D(a)=C(-1/(4a^2),1/(2a))$. To summarize, we have that 
\begin{align*}
K(\eta(a))f_r\otimes \delta_s&=0, \textrm{ if } s\ne -1/(2a), \textrm{ or } r\ne -1/(4a^2),\\
K(\eta(a))f_{-1/(4a^2)}\otimes \delta_{-1/(2a)}&=D(a)f_{-1/(4a^2)}\otimes \delta_{1/(2a)}.
\end{align*}

Note that ${}^\tau\! K(\eta(a))={}^t\! K(\eta(a))$ since ${}^\tau\! \eta(a)=\eta(a)$. Since ${}^\tau\!K+K=0$ and the pair $\pair{~,~}$ on $\sigma_\pi$ is symmetric, we get 
 $ \pair{K(\eta(a))f_{-1/(4a^2)}\otimes \delta_{-1/(2a)},f_{-1/(4a^2)}\otimes \delta_{-1/(2a)}}=0$.  By Eq.(\ref{eq3.7}) and Eq.(\ref{eq5.9}), we have $$\pair{K(\eta(a))f_{-1/(4a^2)}\otimes \delta_{-1/(2a)},f_{-1/(4a^2)}\otimes \delta_{-1/(2a)}}=D(a)\mu^{-1}(x_{-1/(4a^2)}^{q-1}).$$ 
 Thus we get $D(a)=0$, which implies that $K(\eta(a))=0.$

Case (2.3), $\pi=\omega_{\psi_u}^{+}$ for $u=1,\kappa$. Recall that $\omega_{u}^+$ has a basis $\wpair{\Delta_0,\Delta_{x},x\in A_0}$, where $A_0\subset k^\times$ is still a set of representatives of $k^\times/\wpair{\pm 1}$, see Section \ref{sec3.6}. We can check that $\sigma_\pi(\bx_\beta(b))\Delta_r\otimes \delta_s=\psi(b(ur^2+s^2))\Delta_r\otimes \delta_s,$ see Eq.(\ref{eq4.7}) and Eq.(\ref{eq5.7}). Using Eq.(\ref{eq6.1}) when $y=0$, we can get $K(\eta(a))\Delta_r\otimes \delta_s=0$ if $s\ne -1/(2a)$ as above. Applying Eq.(\ref{eq6.2}), we can get $K(\eta(a))\Delta_r\otimes \delta_s=0$ if $ur^2+s^2\ne 0$. Note that there is at most one $r_0\in A_0\cup \wpair{0}$ such that $ur_0^2+1/(4a^2)=0$. If there is no such $r_0$, we are done. Now assume that there exists $r_0$ with $ur_0^2+1/(4a^2)=0$. We assume that $K(\eta(a))\Delta_{r_0}\otimes \delta_{-1/(2a)}=\sum_{b,t}C(b,t)\Delta_b\otimes \delta_t$. Applying Eq.(\ref{eq6.1}) when $x=0$ we get $$\sum_{b,t}\psi(y)C(b,t)\Delta_b\otimes \delta_t=\sum_{b,t}\psi(2ayt)C(b,t)\Delta_b\otimes \delta_t, \forall y\in k.$$ 
By choosing appropriate $y$ we can see $C(b,t)=0$ if $t\ne 1/(2a)$.  Applying Eq.(\ref{eq6.3}) to the equation $K(\eta(a))\Delta_{r_0}\otimes \delta_{-1/(2a)}=\sum_{b}C(b,1/(2a))\Delta_b\otimes \delta_{1/(2a)}$, we can obtain that
$$\sum_b C(b,1/(2a))\psi(x(ub^2+1/(4a^2)))\Delta_b\otimes \delta_{1/2a}=\sum_b C(b,1/(2a))\Delta_b\otimes \delta_{1/(2a)}, \forall x\in k,$$
which implies that $C(b,1/(2a))=0$ if $b\ne r_0$. Thus we have $K(\eta(a))\Delta_{r_0}\otimes \delta_{-1/(2a)}=C(r_0,1/(2a))\Delta_{r_0}\otimes \delta_{1/(2a)} $. From the condition ${}^\tau\!K+K=0$ and the symmetry of the pair $\pair{~,~}$ on $\sigma_\pi$, we have 
$$\pair{K(\eta(a))\Delta_{r_0}\otimes \delta_{-1/(2a)},\Delta_{r_0}\otimes \delta_{-1/(2a)}}=0. $$ By Eq.(\ref{eq3.5}), Eq.(\ref{eq4.9}) and the above discussion, the above equation implies that $C(r_0,1/(2a))=0$. Thus $K(h(a,1)w_\alpha)=0.$ 

Case (2.4), $\pi=\omega_{\psi_u}^-$. The proof is similar to that the Case (2.3) and thus omitted.

Case (2.5), $\pi=\St$. Recall that $\St$ has a basis $\wpair{F_r,r\in k},$ see Section \ref{sec3.5}. By Eqs.(\ref{eq4.5}, \ref{eq5.8}, \ref{eq6.2}), we have 
$$\psi(bs^2)K(\eta(a))F_{r-b}\otimes \delta_s=K(\eta(a))F_r\otimes \delta_s.$$
In particular, we have $K(\eta(a))F_r\otimes \delta_s=\psi(rs^2)K(\eta(a))F_0\otimes \delta_s$. Plugging $y=0$ into Eq.(\ref{eq6.1}) and applying it to $F_0\otimes \delta_s$, we get $\psi(-2sx)K(\eta(a))F_0\otimes \delta_s=\psi(a^{-1}x)K(\eta(a))F_0\otimes \delta_s$ using Eq.(\ref{eq4.5}) and Eq.(\ref{eq5.8}). By choosing appropriate $x$, we see that $K(h(a,1)w_\alpha)F_0\otimes \delta_s=0$ if $s\ne -1/(2a)$. Denote $a_0=-1/(2a)$. 

 We assume that $K(\eta(a))F_0\otimes \delta_{a_0}=\sum_{b,t\in k}C_s(b,t)F_b\otimes \delta_t$. Plugging $x=0$ into Eq.(\ref{eq6.1}) and applying it to $F_0\otimes\delta_{a_0}$, we can get $\sum_{b,t}\psi(y)C_s(b,t)F_b\otimes \delta_t=\sum_{b,t}C_s(b,t)\psi(2ayt)F_b\otimes \delta_t, \forall y\in k.$ By choosing appropriate $y$, we have $C_s(b,t)=0$ if $t\ne -a_0$. Thus we get $K(\eta(a))F_0\otimes \delta_{a_0}=\sum_{b\in k}D_bF_b\otimes \delta_{-a_0},$ where $D_b=C_{a_0}(b,-a_0)$.
 
 Note that ${}^\tau(\eta(a))=\eta(a),$ thus the condition ${}^\tau\! K+K=0$ implies that $$\pair{K(\eta(a))F_r\otimes \delta_{a_0},F_0\otimes \delta_{a_0}}+\pair{F_r\otimes\delta_{a_0},K(\eta(a)F_0\otimes\delta_{a_0})}=0,\forall r\in k.$$
 In particular, we have $\pair{K(\eta(a))F_0\otimes \delta_{a_0},F_0\otimes \delta_{a_0}}=0 $ and thus,
 $$\pair{K(\eta(a))F_r\otimes \delta_{a_0},F_0\otimes \delta_{a_0}}=\psi(ra_0^2)\pair{K(\eta(a))F_0\otimes \delta_{a_0},F_0\otimes \delta_{a_0}}=0.$$
 The above two equations then imply that 
 $$\pair{F_r\otimes\delta_{a_0},K(\eta(a)F_0\otimes\delta_{a_0})}=0,\forall r\in k. $$
  By Eq.(\ref{eq3.4}) and Eq.(\ref{eq5.9}), $\pair{F_r\otimes\delta_{a_0},K(\eta(a)F_0\otimes\delta_{a_0})}=D_{-r}+(q-1)\sum_{b\in k}D_b$. Thus we get $D_{-r}+(q-1)\sum_{b\in k}D_b=0, \forall r\in k.$ A direct calculation shows that $D_b=0,\forall b\in k$. Hence, $K(\eta(a))=0.$ This completes the proof of Step (2).

Step (3), $K(\xi(a))=0$. Recall that $\xi(a)=h(-1,a)w_1$ and $w_1=w_\alpha w_\beta w_\alpha^{-1}$.
We can check the relations
\begin{align*}\xi(a)\bx_\beta(y)&=(0,0,0,0,-ay)\xi(a),\\
\xi(a)(0,0,0,0,-ay)&=\bx_\beta(y)\xi(a),\\
\xi(a) \bx_{\alpha+\beta}(x)&= \bx_{\alpha+\beta}(-x)\xi(a).
\end{align*}
Thus we have 
\begin{align}
\label{eq6.4}K(\xi(a))\sigma_\pi(\bx_\beta(y))&=K(\xi(a)), \forall y\in k,\\
\label{eq6.5} K(\xi(a))&=\sigma_\pi(\bx_\beta(y))K(\xi(a)), \forall y\in k, \\
\label{eq6.6}K(\xi(a))\sigma_\pi(\bx_{\alpha+\beta}(x))&=\sigma_\pi(\bx_{\alpha+\beta}(-x))K(\xi(a)), \forall x\in k.\end{align}
In the following, we still consider different $\pi$ separately. 

Case (3.1), $\pi=1.$
Applying Eq.(\ref{eq6.4}) to $\delta_s$, we get 
$$\psi(ys^2)K(\xi(a))\delta_s=K(\xi(a))\delta_s,\forall y\in k. $$
For $s\ne 0$, we can take $y$ such that $\psi(ys^2)\ne 1$. Thus we have
$$K(\xi(a))\delta_s=0, \textrm{ if }s\ne 0.$$
Suppose that $K(\xi(a))\delta_0=\sum_{s\in k}C(s) \delta_s, $
where $C(s)\in \BC$. Applying Eq.(\ref{eq6.5}) to $\delta_0$, we have
$$ \sum_{s\in k}C(s)\delta_s=K(\xi(a))\delta_0=\sigma_\pi(\bx_\beta(y))K(\xi(a))\delta_0=\sum_{s\in k} \psi(ys^2)C(s)\delta_s, \forall y\in k.$$
If $s\ne 0$, take $y\ne 0$ such that $\psi(ys^2\ne 0)$. By comparing coefficients,  we have $C(s)=0$ unless $s=0$. Therefore,
$$K(\xi(a))\delta_s=0, s\ne 0; K(\xi(a))\delta_0=C(0)\delta_0.$$
From the relation ${}^\tau \! (\xi(a))=\xi(a)$, ${}^\tau\! K+K=0$ and the symmetry of the pair $\pair{~}$ on $\sigma_\pi$, we get $\pair{ K(\xi(a))\delta_0,\delta_0}=0$.  By Eq.(\ref{eq5.9}), we have $$ \pair{ K(\xi(a))\delta_0,\delta_0}=\pair{C(0)\delta_0,\delta_0}=C(0).$$ It follows that $C(0)=0$ and $K(\eta(a))=0$.

Case (3.2), $\pi=\omega_{\psi,\mu}$. Applying Eq.(\ref{eq6.4}) to $f_r\otimes\delta_s$, we have $$\psi(y(r+s^2))K(\xi(a))f_r\otimes \delta_s=K(\xi(a))f_r\otimes \delta_s.$$ This implies that $K(\xi(a))f_r\otimes \delta_s=0$ if $r+s^2\ne 0$. Suppose that $K(\xi(a))f_{-s^2}\otimes \delta_s=\sum_{b,t}C_s(b,t)f_b\otimes \delta_t.$ Applying Eq.(\ref{eq6.6}) to $f_{-s^2}\otimes\delta_s$, we have 
$$\sum_{b,t}\psi(-2sx)C_s(b,t)f_b\otimes \delta_t=\sum_{b,t}\psi(2tx)C_s(b,t)f_b\otimes \delta_t, \forall x\in k.$$
An appropriate choice of $x\in k$ implies that $C_s(b,t)=0$ if $t\ne -s$. We get $K(\xi(a))f_{-s^2}\otimes \delta_s=\sum_b C_s(b,-s)f_b\otimes \delta_{-s}.$ Applying Eq.(\ref{eq6.5}) to $f_{-s^2}\otimes\delta_s$, we have
$$ \sum_b C_s(b,-s)f_b\otimes\delta_{-s}=\sum_b C_s(b,-s)\psi(y(b+s^2))f_b\otimes\delta_{-s},\forall y\in k,$$
which implies that $C_s(b,-s)=0$ unless $b=-s^2$. Thus we get $K(\xi(a))f_{-s^2}\otimes \delta_s=D(s)f_{-s^2}\otimes \delta_{-s}$, where $D(s)=C_s(-s^2,-s)$. The condition ${}^\tau\!K+K=0$ implies that $$ \pair{f_{-s^2}\otimes \delta_{s},K(\xi(a))f_{-s^2}\otimes \delta_s}=0.$$  By Eqs.(\ref{eq3.7}) and (\ref{eq5.9}), the above equation implies that $D(s)=0$. Hence, we have $K(\xi(a))=0.$

Case (3.3), $\pi=\omega_{\psi_u}^{\pm}$ for $u=1,\kappa$. The proof is similar to that of Case (3.2) and thus omitted.

Case (3.4), $\pi=\St$. As in the proof of case (2.5), an application of Eq.(\ref{eq6.4}) to $F_r\otimes \delta_s$ shows that $K(\xi(a))F_r\otimes \delta_s=\psi(rs^2)K(\xi(a))F_0\otimes \delta_s$. Suppose that $K(\xi(a))F_0\otimes\delta_s=\sum_{b,t}C_s(b,t)F_b\otimes \delta_t$ for $C_s(b,t)\in \BC$. Applying Eq.(\ref{eq6.6}) to $F_0\otimes\delta_s$ and using Eqs.(\ref{eq4.5}), (\ref{eq5.8}),  we get 
$$\sum_{b,t}\psi(-2xs)C_s(b,t)F_b\otimes \delta_t=\sum_{b,t}C_s(b,t)\psi(2xt)F_b\otimes \delta_t,\forall x\in k.$$
If $t\ne -s$, we can choose $x\in k$ such that $\psi(-2xs)\ne \psi(2xt)$. Consequently, $C_s(b,t)=0$ unless $t=-s$, and $K(\xi(a))F_0\otimes\delta_s=\sum_{b\in k}C_s(b)F_b\otimes \delta_{-s}$, where $C_s(b)=C_s(b,-s)$. An application of Eq.(\ref{eq6.5}) to $F_0\otimes\delta_s$ shows that
$$\sum_{b\in k}C_s(b)F_b\otimes \delta_{-s}=\sum_{b_1\in k}C_s(b_1)\psi(ys^2)F_{b_1-y}\otimes \delta_{-s},\forall y\in k.$$ 
By comparing the coefficients of both sides of the above identity, we get $C_s(b)=\psi(ys^2)C_s(b+y).$ In particular, we have $C_s(b)=\psi(-bs^2)C_s(0)$. Thus $K(\xi(a))F_0\otimes \delta_s=C_s(0)\sum_{b\in k}\psi(-bs^2)F_b\otimes \delta_{-s}.$ As usual, the condition ${}^\tau\!K+K=0$ implies that $\pair{K(\xi(a))F_0\otimes \delta_s,F_0\otimes \delta_s}=0$. By Eq.(\ref{eq3.4}) and Eq.(\ref{eq5.9}), we get
$$C_s(0)\left((q-1)\sum_{b\in k}\psi(-bs^2)+\psi(0)\right)=\pair{K(\xi(a))F_0\otimes \delta_s,F_0\otimes \delta_s}=0.$$ Because $\sum_{b\in k}\psi(-bs^2)=0 $ if $s\ne 0$, and $\sum_{b\in k}\psi(-bs^2)=q$ if $s=0$, we have $C_s(0)=0$. Hence $K(\xi(a))F_0\otimes\delta_s=0$ and $K(\xi(a))F_r\otimes \delta_s=\psi(rs^2)K(\xi(a))F_0\otimes \delta_s=0$. This shows $K(\xi(a))=0$ and completes the proof of Step (3).

Step (4),  $K(h(a,1)w_2)=0$.
We have 
$$h(a,1)w_2 g=g^a h(a,1)w_2,$$
where $$ g^a=\begin{pmatrix}x_{11}& ax_{12}\\a^{-1} x_{21}& x_{22} \end{pmatrix}, \textrm{ for } g=\begin{pmatrix}x_{11}& x_{12}\\ x_{21}& x_{22} \end{pmatrix}\in \SL_2(k).$$
Thus,
$$K(h(a,1)w_2)\sigma_\pi(g)=\sigma_\pi^a(g)K(h(a,1)w_2),\forall g\in \SL_2(k),$$
where $\sigma_\pi^a(g)=\sigma_\pi(g^a).$ The above equation implies that
$$K(h(a,1)w_2)\in \Hom_{\SL_2(k)}(\sigma_\pi,\sigma_\pi^a).$$
We first assume that $a\notin k^{\times,2}$. 

If $\pi=1$, we have $\sigma_\pi|_{\SL_2}=\omega_\psi^+\oplus \omega_\psi^-$ and $\sigma_\pi^a|_{\SL_2(k)}\cong \omega_{\psi_\kappa}^+\oplus \omega_{\psi_\kappa}^-$. We have $$\Hom_{\SL_2(k)}(\sigma_\pi,\sigma_\pi^a)=0,$$ and thus $K(h(a,1)w_2)=0$.

If $\pi=\omega_{\psi,\mu}, \omega_{\psi_u}^{\pm 1}$ for $u=1,\kappa$, then $\Hom_{\SL_2(k)}(\sigma_\pi,\sigma_\pi^a)\ne 0$. In fact, from the decomposition given in Proposition \ref{prop2.1}, we can write $\sigma_\pi|_{\SL_2(k)}=\bigoplus_{i\in I}V_i$, where $I$ is an index set and each $V_i$ is an irreducible representation of $\SL_2(k)$ and occurred with multiplicity one. If $V_i$ is one of $1,I(\chi),\omega_{\psi,\mu}, \St$, then $V_i\cong V_i^a$. By Lemma \ref{lem6.4}, we fix an isomorphism $\lambda_{V_i}^a\in \Hom_{\SL_2(k)}(V_i,V_i^a)$ such that $$\pair{\lambda_{V_i}^a v_1,\lambda_{V_i}^av_2}=\pair{v_1,v_2},\forall v_1,v_2\in V_2.$$ 
If $V_i\cong \omega_\psi^{\pm}$ (resp. $\omega_{\psi_\kappa}^{\pm}$), then $V_i^a\cong \omega_{\psi_\kappa}^{\pm}$ (resp. $\omega_{\psi}^{\pm}$). From the decomposition of $\sigma_\pi|_{\SL_2(k)}$, one sees that 
$$\Hom_{\SL_2(k)}(\sigma_\pi,\sigma_\pi^a)=\bigoplus_{i\in I_1}\Hom(V_i,V_i^a),$$
where $I_1$ is the subset of $I$ such that $V_i$ is $1,\St,I(\chi)$ or $\omega_{\psi,\mu}$ for $i\in I_1$. For example, when $\pi=\omega_{\psi,\mu_1}$ and $\epsilon_0=1$, there are factors $\omega_{\psi_\kappa}^+\oplus \omega_{\psi_\kappa}^-$ in the decomposition $\sigma_\pi|_{\SL_2(k)}$, but these factors are not indexed by $I_1$, in fact, one has $\Hom_{\SL_2(k)}(\omega_{\psi_\kappa}^+\oplus \omega_{\psi_\kappa}^-,(\omega_{\psi_\kappa}^+\oplus \omega_{\psi_\kappa}^-)^a )=0.$ Hence, there are constants $C_i\in \BC$ such that 
$$K(h(a,1)w_2)=\sum_{i\in I_1}C_i\lambda_{V_i}^a.$$ Note that ${}^\tau\! (h(a,1)w_2)=h(1,a)w_2=h(a^{-1},a)h(a,1)w_2$. Thus
$$K(h(1,a)w_2)=\sum_i C_i\sigma_\pi(h(a^{-1},a))\lambda_{V_i}^a,$$ and the condition $K=-{}^\tau\!K$ implies that $K(h(a,1)w_2)=-({}^\tau\!K)(h(a,1)w_2)=-{}^t\! K(h(1,a)w_2)$. From the definition of the transpose operator, we get $$\pair{v_1,K(h(a,1)w_2)v_2}+\pair{K(h(1,a)w_2)v_1,v_2}=0,\forall v_1,v_2\in \pi\otimes\omega_\psi.$$
In particular, if we choose $v_1,v_2\in V_i$ for a fixed $i\in I_1$, we have 
$$ C_i\left(\pair{v_1,\lambda_{V_i}^a v_2}+\pair{\sigma_\pi(h(a^{-1},a))\lambda_{V_i}^a v_1,v_2}\right)=0,\forall v_1,v_2\in V_i.$$ If we replace $v_1$ by $\lambda_{V_i}^a v_1$, we get 
$$ C_i\left(\pair{\lambda_{V_i}^a v_1,\lambda_{V_i}^a v_2}+\pair{\sigma_\pi(h(a^{-1},a))\lambda_{V_i}^a\lambda_{V_i}^a v_1,v_2}\right)=0,\forall v_1,v_2\in V_i.$$
By Lemma \ref{lem6.4}, we have $\sigma_\pi(h(a^{-1},a))\lambda_{V_i}^a\circ \lambda_{V_i}^a=\id_{V_i}$.   We then get 
$$ C_i\left(\pair{\lambda_{V_i}^a v_1,\lambda_{V_i}^a v_2}+\pair{ v_1,v_2}\right)=0,\forall v_1,v_2\in V_i,$$
i.e.,  
$$2C_i\pair{v_1,v_2}=0,\forall v_1,v_2\in V_i.$$
This implies that $C_i=0$ for all $i$ and thus $K(h(a,1)w_2)=0$. 

Next, we assume that $a\notin k^{\times,2}$ and $\pi=\St$. Without loss of generality, we assume that $a=\kappa^{-1}$. The proof is almost identical to the above case, but for completeness we still provide the details here.  Note that we have that $\sigma_\pi|_{\SL_2(k)}\cong \sigma_\pi^a|_{\SL_2(k)}$. We consider the decomposion $\sigma_\pi|_{\SL_2(k)}$ given in Proposition \ref{prop2.1}:
$$\St\otimes \omega_\psi|_{\SL_2(k)}=\St\bigoplus\left(\bigoplus_{\chi\in A}I(\chi)\right)\bigoplus\left(\bigoplus_{\mu\in B}\omega_{\psi,\mu} \right)\bigoplus \omega_{\psi}^+ \bigoplus \omega_{\psi_\kappa}^+.$$ 
Let $V$ be a summand in the above decomposition. If $V=\St, I(\chi)$ or $\omega_{\psi,\mu}$, we have $V\cong V^a$. In these cases, by Lemma \ref{lem6.4}, we fix a nonzero $\lambda_V^a\in \Hom_{\SL_2(k)}(V,V^a)$ such that $$\pair{\lambda_V^av_1,\lambda_V^av_2}=\pair{v_1,v_2},\forall v_1,v_2\in V.$$ If $V=\omega_\psi^+$, then $V^a\cong \omega_{\psi_\kappa}^+$; and if $V=\omega_{\psi_\kappa}^+$, then $V^a\cong \omega_\psi^+$. In fact, a direct calculation shows that $(\omega_{\psi_\kappa})^a$ is exactly $\omega_\psi$ since we assumed $a=\kappa^{-1}$ (if $a=a_0^2\kappa^{-1}$ for some $a_0\in k^\times$, the two representations $(\omega_{\psi_\kappa})^a$ and $\omega_\psi$ would differ by an inner automorphism).  Let $\lambda_{\omega_\psi^+}^a\in \Hom(\omega_\psi^+,(\omega_{\psi_\kappa}^+)^a)$ and $\lambda_{\omega_{\psi_\kappa}^+}^a\in \Hom_{\SL_2(k)}(\omega_{\psi_\kappa}^+,(\omega_{\psi}^+)^a)$ be the identity maps. 

From the decomposition of $\St\otimes \omega_\psi|_{\SL_2(k)}$, 
we have
\begin{align*}\Hom_{\SL_2(k)}(\sigma_\pi,\sigma_\pi^a)&= \left(\bigoplus_{\chi\in A}\Hom_{\SL_2(k)}(I(\chi),I(\chi)^a)\right)
 \bigoplus\left(\bigoplus_{\mu\in B}\Hom_{\SL_2(k)}(\omega_{\psi,\mu},\omega_{\psi,\mu}^a)\right)\\
&\bigoplus\Hom_{\SL_2(k)}(\St,\St^a)\bigoplus \Hom_{\SL_2(k)}(\omega_\psi^+,(\omega_{\psi_\kappa}^+)^a)\bigoplus\Hom_{\SL_2(k)}(\omega_{\psi_\kappa}^+,(\omega_{\psi}^+)^a).
\end{align*}
By Schur's Lemma, each Hom space on the right hand side of the above equation has dimension 1 and is generated by $\lambda_V^a$ for the corresponding $V$ in the decomposition of $\St\otimes\omega_\psi|_{\SL_2(k)}$. Thus we can write $$K(h(a,1)w_2)=\sum_{V}C_V \lambda_V^a,$$
for some constants $C_V\in \BC$.

Since ${}^\tau\!(h(a,1)w_2)=h(1,a)w_2=h(a^{-1},a)h(a,1)w_2$ and $h(a^{-1},a)\in J$, we have that  $$K(h(1,a)w_2)=\sigma_\pi(h(a^{-1},a))K(h(a,1)w_2).$$
Let $\lambda_V^{a^{-1}}=\sigma_\pi(h(a^{-1},a))\lambda_V^a$, which is an isomorphism in $\Hom_{\SL_2(k)}(V,V^{a^{-1}})$ if $V=\St,I(\chi)$ or $\omega_{\psi,\mu}$, and is an isomorphism in $\Hom_{\SL_2(k)}(\omega_\psi^+,(\omega_{\psi_\kappa}^+)^{a^{-1}})$ (resp. $\Hom_{\SL_2(k)}(\omega_{\psi_\kappa}^+,(\omega_{\psi}^+)^{a^{-1}}) $) if $V=\omega_\psi^+$ (resp. $\omega_{\psi_\kappa}^+$). Then $$K(h(1,a)w_2)=\sum_V C_V \lambda_V^{a^{-1}}.$$ Since $K$ is $\tau$-skew-invariant, 
$$K(h(a,1)w_2)=-({}^\tau\!K)(h(a,1)w_2)=-{}^t\! K(h(1,a)w_2).$$ By the definition of the transpose, 
$$\pair{v_1,K(h(a,1)w_2)v_2}+\pair{K(h(1,a)w_2)v_1,v_2}=0,\forall v_1,v_2\in \St\otimes\omega_\psi.$$
In particular, choosing $v_1,v_2\in V$, where $V$ is still a component of $\sigma_\pi|_{\SL_2(k)}$, we have that 
\begin{equation}\label{equation:C_V}
    \qquad C_V\left(\pair{v_1,\lambda_V^a v_2}+\pair{\lambda_V^{a^{-1}}v_1,v_2}\right)=0,\forall v_1,v_2\in V.
\end{equation}
 As in the previous case, one can show that $C_V=0$ for each $V\ne \omega_{\psi}^+$ or $\omega_{\psi_\kappa}^+$. If $V=\omega_\psi^+$ or $\omega_{\psi_\kappa}^+$, then $\lambda_V^a$ is simply the identity map and $\lambda_V^{a^{-1}}=\sigma_\pi(h(a^{-1},a))$. Taking $v_1=\Delta_s=v_2$ for some $s\in k^\times$ in Eq.(\ref{equation:C_V}), then  $\lambda_V^{a^{-1}}\Delta_s=\epsilon(\kappa)\Delta_{\kappa^{-1}s}=-\Delta_{\kappa^{-1}s}$. Thus we have $\pair{v_1,\lambda_V^a v_2}=2 $ and $\pair{\lambda_V^{a^{-1}}v_1,v_2}=0$. Then Eq.(\ref{equation:C_V}) implies that $C_V=0$. Thus $K(h(a,1)w_2)=0$. 

At last, we assume that $a=a_0^2$, $a_0 \in k^{\times}$. Since $h(a,1)w_2=h(a_0,a_0^{-1})h(a_0,a_0)w_2$ and $h(a_0,a_0^{-1})\in J $, we have 
$$K(h(a,1)w_2)=\sigma_\pi(h(a_0,a_0^{-1}))K(h(a_0,a_0)w_2).$$
Thus it suffices to show $K(h(a_0,a_0)w_2)=0.$ Note that we have the relation 
$$h(a_0,a_0)w_2g=gh(a_0,a_0)w_2, \forall g\in \SL_2(k),$$
which implies that $$K(h(a_0,a_0)w_2)\sigma_\pi(g)=\sigma_\pi(g) K(h(a_0,a_0)w_2),\forall g\in \SL_2(k).$$
Thus $K(h(a_0,a_0)w_2)\in \Hom_{\SL_2(k)}(\sigma_\pi,\sigma_\pi)$. Let $\sigma_\pi|_{\SL_2(k)}=\oplus V_i$ be the decomposition as in Proposition \ref{prop2.1}. We have $\Hom_{\SL_2(k)}(\sigma_\pi,\sigma_\pi)=\oplus \End_{\SL_2(k)}(V_i) $. Thus
$$K(h(a_0,a_0)w_2)=\sum C_i \id_{V_i}.$$
Note that ${}^\tau \! (h(a_0,a_0)w_2)=h(a_0,a_0)w_2,$ and ${}^t \id_{V_i}=\id_{V_i}$ by  the discussion in Section \ref{sec5.5}. Thus 
$$({}^\tau \! K)(h(a_0,a_0)w_2)=
{}^t\! K(h(a_0,a_0)w_2)=\sum_i C_i{}^t \id_{V_i}=K(h(a_0,a_0)w_2).$$
The condition ${}^\tau \! K=-K$ implies that $K(h(a_0,a_0)w_2)=0.$ This completes the proof of Theorem \ref{thm6.1}.
\qed

\end{document}